\documentclass[12pt]{amsart}
\usepackage{amsmath,amssymb,amsfonts,amsthm,amscd,indentfirst}
\usepackage{amsmath,amssymb,amsfonts,amsthm,amscd}
\usepackage{indentfirst}
\usepackage{hyperref}

\numberwithin{equation}{section}

\newtheorem{prop}{Proposition}[section]
\newtheorem{theo}[prop]{Theorem}
\newtheorem{lemm}[prop]{Lemma}

\newtheorem{rema}[prop]{Remark}

\newtheorem{defi}[prop]{Definition}

\def\and{\quad{\rm and}\quad}

\def\<{\langle}
\def\>{\rangle}

\usepackage{graphicx}

\begin{document}
\title[Hessian quotient equation]{Pogorelov type interior $C^2$ estimate for Hessian quotient equation and its application}
\author[Siyuan Lu and Yi-Lin Tsai]{Siyuan Lu and Yi-Lin Tsai}
\address{Department of Mathematics and Statistics, McMaster University, 1280
Main Street West, Hamilton, ON, L8S 4K1, Canada.}
\email{siyuan.lu@mcmaster.ca}
\email{tsaiy11@mcmaster.ca}
\thanks{Research of the first author was supported in part by NSERC Discovery Grant.}

\begin{abstract}
In this paper, we derive a Pogorelov type interior $C^2$ estimate for the Hessian quotient equation $\frac{\sigma _n}{\sigma _k}\left( D^2u\right) =f$. As an application, we show that convex viscosity solutions are regular for $k\leq n-3$ if $u\in C^{1,\alpha}$ with $\alpha>1-\frac{2}{n-k}$ or $u\in W^{2,p}$ with $p\geq\frac{(n-1)(n-k)}{2}$. Both exponents are sharp in view of the example in \cite{Lu24}.
\end{abstract}

\maketitle

\section{Introduction}

In a classic paper \cite{H}, Heinz proved a pure interior $C^2$ estimate for the Monge-Amp\`ere equation $\det\left( D^2u\right)=f$ for $n=2$ (see also recent proofs \cite{CHO,Liu}). However, such estimate fails for the Monge-Amp\`ere equation for $n\geq 3$, due to a counterexample by Pogorelov \cite{Pb}. Instead, Pogorelov \cite{Pb} proved an interior $C^2$ estimate with homogeneous boundary conditions for the Monge-Amp\`ere equation for all dimensions (see also recent proofs \cite{Yuan,SY-24}). This estimate is known as \textit{Pogorelov type interior $C^2$ estimate} in the literature and it plays an essential role in the regularity theory of Monge-Amp\`ere equation.

It is of great interest to determine whether Pogorelov type interior $C^2$ estimate holds for other fully nonlinear equations. 

A natural class of fully nonlinear equations is the Hessian equation $\sigma_k\left( D^2u\right)=f$, where $\sigma_k$ is the $k$-th elementary symmetric function. Note that for $k=n$, it reduces to the Monge-Amp\`ere equation, and for $k=1$, it reduces to the Laplace equation. In a fundamental paper \cite{CW}, Chou and Wang established a Pogorelov type interior $C^2$ estimate for the Hessian equation $\sigma_k\left( D^2u\right)=f$. For recent progress, see \cite{LRW,Tu,RZhang}.

Another natural class of fully nonlinear equations is the Hessian quotient equation. In contrast to the Hessian equation, very little was known in this case. The purpose of this paper is to establish a Pogorelov type interior $C^2$ estimate for the following Hessian quotient equation $\frac{\sigma_n}{\sigma_k}\left( D^2u\right)=f$.

Before we state our main theorem, let us recall the definition of section, see for instance in \cite{CG}.
\begin{defi}
Let $u$ be a convex function with $u(0)=0$ and $Du(0)=0$. For $h>0$, define the section $\Sigma_h$ of $u$ by 
\begin{align*}
\Sigma_h=\{x\in \mathbb{R}^n|u(x)<h\}.
\end{align*}
\end{defi}

In the following, when we write $u\in C^{k}\left( \overline{\Sigma }_{h}\right) $, we will always assume that $u(0)=0$, $Du(0)=0$, and $u=h$ on $\partial \Sigma _{h}$.

\begin{theo}\label{theorem-1} 
Let $n\geq 2$, $1\leq k<n$. Let $f\in C^{1,1}\left( \overline{\Sigma} _2\times \mathbb{R}\right)$ be a positive function and let $u\in C^4\left(\overline{\Sigma}_2\right) $ be a convex solution of 
\begin{equation}\label{main-eqn}
\frac{\sigma _n}{\sigma _k}\left( D^2u\right)=f(x,u).
\end{equation}

Then there exist positive constants $\tau$ and $C$, depending only on $n$, $k$, the diameter of $\Sigma _2$, $\|u\| _{C^{0,1}\left(\overline{\Sigma}_2\right) }$, $\min_{\overline{\Sigma}_2\times [0,2]} f$ and $\|f\|_{C^{1,1}\left( \overline{\Sigma}_2\times [0,2]\right) }$, such that $\overline{B}_\tau \subseteq \Sigma _1$ and 
\begin{equation*}
\max_{\overline{B}_\tau}|D^{2}u|\leq C.
\end{equation*}
\end{theo}

As an application, we will show that $C^{1,\alpha}$ or $W^{2,p}$ viscosity solutions are regular by following the work of Urbas \cite{Urbas88}, and Collins and Mooney \cite{Mooney-Collins} on the Monge-Amp\`ere equation. 

Let us recall the definition of a viscosity solution, see for instance in \cite{CIL}.
\begin{defi}
Let $n\geq 2$, $1\leq k<n$ and let $\Omega$ be a bounded domain. Let $f\in C^0(\Omega)$ be a positive function and let $u\in C^0(\Omega)$ be a locally convex function.  We say $u$ is a viscosity subsolution of 
\begin{align}\label{viscosity}
\frac{\sigma _n}{\sigma _k}\left(D^{2}u\right) =f(x),\quad in\quad \Omega,
\end{align}
if for any $x_0\in \Omega $, $\varphi \in C^2( \Omega)$ satisfying
\begin{align*}
\varphi(x_0)=u(x_0),\quad and\quad \varphi\geq u,\quad in\quad\Omega,
\end{align*}
we have
\begin{equation*}
\frac{\sigma _n}{\sigma _k}\left(D^2\varphi(x_0)\right) \geq  f(x_0).
\end{equation*}

Similarly, we say $u$ is a viscosity supersolution of (\ref{viscosity}) if for any $x_0\in \Omega $, $\varphi \in C^2(\Omega)$ satisfying
\begin{align*}
\varphi(x_0)=u(x_0),\quad and\quad \varphi\leq u,\quad in\quad\Omega,
\end{align*}
we have
\begin{equation*}
\frac{\sigma _n}{\sigma _k}\left(D^2\varphi(x_0)\right) \leq  f(x_0).
\end{equation*}

We say $u$ is a viscosity solution of (\ref{viscosity}) if $u$ is both a viscosity subsolution and a viscosity supersolution of (\ref{viscosity}).
\end{defi}

\begin{theo}
\label{theorem-v}
Let $n\geq 4$, $1\leq k\leq n-3$ and let $\Omega$ be a bounded domain. Let $f\in C^{1,1}\left( \overline{\Omega}\right) $ be a positive function and let $u$ be a viscosity solution of (\ref{viscosity}). Suppose one of the following holds:
\begin{enumerate}
\item $u\in C^{1,\alpha }\left( \overline{\Omega}\right) $ for some $\alpha >1-\frac{2}{n-k}$;
\item $u\in W^{2,p}( \Omega ) $ for some $p\geq\frac{( n-1) ( n-k) }{2}$.
\end{enumerate}

Then $u\in C^{3,\beta }( \Omega) $ for any $0<\beta <1 $. Moreover, for any $\Omega^\prime\Subset \Omega$ , we have
\begin{equation*}
\| u\| _{C^{3,\beta }\left( \overline{\Omega}^\prime\right) }\leq C,  
\end{equation*}
where $C$ is a positive constant depending only on $n, k,\alpha ,\beta ,\Omega^\prime,\Omega ,\min_{\overline{\Omega}}f,\|f\| _{C^{1,1}\left( \overline{\Omega }\right)}$ and $\| u\| _{C^{1,\alpha}\left( \overline{\Omega}\right) }$ in case $(1)$, and on $n,k,p,\beta ,\Omega^\prime,\Omega ,\min_{\overline{\Omega}}f,\|f\| _{C^{1,1}\left( \overline{\Omega }\right) }$ and $\| u\| _{W^{2,p}( \Omega) }$ in case $(2)$.
\end{theo}

\begin{rema}
\begin{enumerate}
\item We remark that the conditions on $\alpha$ and $p$ are sharp, in view of the example in \cite{Lu24}.
\item Previously, case $(2)$ was proved by Bao, Chen, Guan and Ji \cite{BCGJ} under the non-sharp condition $p>(n-1)(n-k)$. For related works, see \cite{BC,LB}.
\end{enumerate}
\end{rema}

\begin{rema}
For Monge-Amp\`ere equation, Theorem \ref{theorem-v} is proved by Urbas \cite{Urbas88} for $\alpha>1-\frac{2}{n}$ and $p>\frac{n(n-1)}{2}$. The borderline case $p=\frac{n(n-1)}{2}$ was proved by Collins and Mooney \cite{Mooney-Collins}.
\end{rema}

The classic way to prove Pogorelov type interior $C^2$ estimate is to choose a suitable test function, and to obtain the desired estimate via maximum principle. Since $u$ is constant on the boundary, it is standard to choose $u$ as the cut-off function. This strategy works out perfectly for the Monge-Amp\`ere equation as well as the Hessian equation.

However, such method does not work for the Hessian quotient equation. A key difference is that for Hessian equation, the terms $\sum_i\sigma_k^{ii}$ and $\sum_i\sigma_k^{ii}u_{ii}^2$ will generate a positive term involving $\lambda_1$, where $\lambda_1$ is the largest eigenvalue of $D^2u$. These two terms are the major terms to control the negative terms generating from the cut off function. In contrast, for Hessian quotient equation, the terms $\sum_iF^{ii}$ and $\sum_iF^{ii}u_{ii}^2$ will not generate a positive term involving $\lambda_1$. Therefore, we can not control the negative terms generating from the cut off function by these two terms. Interested readers may refer to \cite{SUW} for detailed discussions.

Instead, we will adopt the method in \cite{Lu24}, which is inspired by the work of Shankar and Yuan \cite{SY-20}. Let $b=\ln \lambda _1$, where $\lambda _1$ is the largest eigenvalue of $D^2u$. It suffices to bound $b$ in the interior of $\Sigma_2$.

The first step is to prove a Jacobi inequality for $b$. The second step is to prove a mean value inequality to bound $b$ via its integral using Legendre transform. The third step is to estimate the integral via integration by parts. 

For step one, the main difficulty is to prove a concavity inequality for Hessian quotient operator. This was proved recently by Guan and Sroka \cite{GS}. We are very fortunate to be able to use their result here.

Our major difficulty occurs in step two and three. When $k=n-1, n-2$, using Legendre transform, one can easily check that $b$ is a subsolution of a uniformly elliptic equation. Consequently, we can use mean value inequality to bound $b$ via its integral, see in \cite{Lu24}. However, for general $k$, this is not true without extra conditions.

What saves us is the homogeneous boundary condition. Our key observation is that after applying the Legendre transform, we obtain a Hessian equation with the boundary condition $u^*=0$. Although this is not the usual homogeneous boundary condition, we manage to obtain a Pogorelov type interior $C^2$ estimate in this setting. This estimate implies that $b$ is a subsolution of a uniformly elliptic equation. It also implies the eigenvalue of $D^2u$ is uniformly bounded below. See Lemma \ref{Pogolev-1} for details. Both properties are crucial in step two and three. We would like to remark that lower bound estimate is usually much harder to obtain. Our method provides a new way to obtain such estimate.

We would like to point out that the boundary condition is necessary in our theorem. In fact, the pure interior $C^2$ estimate fails for $k\leq n-3$, see example in \cite{Lu24}.

\medskip

Before we end the introduction, let us briefly mention the history of the pure interior $C^2$ estimate for Hessian and Hessian quotient equations.

For the Hessian equation, Urbas \cite{Urbas90} showed that the pure interior $C^2$ estimate fails for the Hessian equation $\sigma _k=f$ for $k\geq 3$ by generalizing Pogorelov's result. The remaining case $\sigma _2=f$ turns out to be extremely difficult. In a breakthrough \cite{WY09}, Warren and Yuan proved an interior $C^2$ estimate for $\sigma _2=1$ for $n=3$. Recently, Qiu \cite{Q1,Q2} obtained an interior $C^2$ estimate for $\sigma _2=f$ for $n=3$. Most recently, in a cutting-edge paper \cite{SY-25}, Shankar and Yuan established an interior $C^2$ estimate for $\sigma_2=1$ for $n=4$. For general dimensions, under various convexity assumptions, interior $C^2$ estimate for $\sigma_2$ equation was proved by Guan and Qiu \cite{GQ}, McGonagle, Song and Yuan \cite{MSY}, Mooney \cite{Mooney}, and Shankar and Yuan \cite{SY-20,SY-25}. See also recent progress \cite{CJZ, Zhou}. We would like to mention that the interior $C^2$ estimate for $\sigma_2$ equation remains open for $n\geq 5$.

For Hessian quotient equation, interior $C^2$ estimate for $\frac{\sigma _3}{\sigma _1}=f$ in dimension $3$ and $4$ was obtained by Chen, Warren and Yuan \cite{CWY}, Wang and Yuan \cite{WdY14}, and Zhou \cite{Zhou} using special Lagrangian structure of the equation. In recent papers \cite{Lu23, Lu24}, the first author proved an interior $C^2$ estimate for $\frac{\sigma _n}{\sigma _k}=f$ for $n-k\leq 2$ via Legendre transform. For general Hessian quotient equation $\frac{\sigma_k}{\sigma_l}=f$, the first author \cite{Lu24} constructed a singular solution for $k-l\geq 3$ by extending Pogorelov's example. We would like to mention that interior $C^2$ estimate for the general Hessian quotient equation $\frac{\sigma _k}{\sigma _l}=f$ is largely open when $k-l\leq 2$ and $k\leq n-1$.

\medskip

The organization of the paper is as follows. In Section 2, we will collect some basic properties of the Hessian and Hessian quotient operators as well as Legendre transform. In Section 3, we will state some important lemmas, which play essential roles in the a priori estimates. In Section 4, we will prove the strict positivity of the Hessian. This is achieved by studying the equation via Legendre transform. This is a key ingredient of our proof. In Section 5, we will perform a Legendre transform to bound $b$ by its integral. In Section 6, we will complete the proof of Theorem \ref{theorem-1} via integration by parts. In Section 7, we will prove strict convexity of viscosity solution. In Section 8, we will prove Theorem \ref{theorem-v}.

\section{Preliminaries}

In this section, we will collect some basic formulas for Hessian and Hessian quotient equations.

\subsection{Hessian and Hessian quotient operators}

In this subsection, we will collect some basic properties of Hessian and Hessian quotient operators.

Let $\lambda=(\lambda_1,\cdots,\lambda_n)\in \mathbb{R}^n$, we will denote 
\begin{align*}
(\lambda|i)=(\lambda_1,\cdots,\lambda_{i-1},\lambda_{i+1},\cdots,\lambda_n)\in \mathbb{R}^{n-1},
\end{align*}
i.e. $(\lambda|i)$ is the vector obtained by deleting the $i$-th component of the vector $\lambda$. Similarly, $(\lambda|ij)$ is the vector obtained by deleting the $i$-th and $j$-th components of the vector $\lambda$.

We now collect some basic formulas for Hessian operator, see for instance \cite{LT}.

\begin{lemm}
\label{Sigma_k-Lemma-0} For any $\lambda=(\lambda_1,\cdots,\lambda_n)\in 
\mathbb{R}^n$, we have 
\begin{align*}
\sigma_k(\lambda)=\lambda_i\sigma_{k-1}(\lambda|i)+\sigma_k(\lambda|i),\quad \sum_i\sigma_{k}(\lambda|i)&\ =(n-k)\sigma_{k}(\lambda), \\
\sum_i\lambda_i\sigma_{k-1}(\lambda|i)= k\sigma_k(\lambda),\quad \sum_i\lambda_i^2\sigma_{k-1}(\lambda|i)=&\ \sigma_1\sigma_k-(k+1)\sigma_{k+1}.
\end{align*}
\end{lemm}

For simplicity, we will use $\sigma_k(|i)$ to denote $\sigma_k(\lambda|i)$ if there is no ambiguity. Similarly, we will use $\sigma_k(|ij)$ to denote $\sigma_k(\lambda|ij)$.

\medskip

Define Garding's $\Gamma_k$ cone by 
\begin{align*}
\Gamma_k=\{\lambda\in \mathbb{R}^n|\sigma_j(\lambda)>0,1\leq j\leq k\}.
\end{align*}

We say $D^2u\in \Gamma_k$ if $\lambda\in \Gamma_k$, where $\lambda=(\lambda_1,\cdots,\lambda_n)$ are the eigenvalues of $D^2u$.

\medskip

Let $F$ be an operator defined on $\lambda$, where $\lambda=(\lambda_1,\cdots,\lambda_n)$ are the eigenvalues of $D^2u$. Then we can extend $F$ such that it is defined on $D^2u$ and we denote 
\begin{align*}
F^{pq}=\frac{\partial F}{\partial u_{pq}},\quad F^{pq,rs}=\frac{\partial^2F}{\partial u_{pq}\partial u_{rs}}.
\end{align*}

We have the following basic properties of the Hessian operator,

\begin{lemm}\label{F^ijkl-1} 
Let $n\geq 2$, $1< k\leq n$ and let $D^2u\in \Gamma_k$. Suppose $D^2u$ is diagonalized at $x_0$. Then at $x_0$, we have 
\begin{align*}
\sigma_k^{pq}=\sigma_k^{pp}\delta_{pq},\quad \sigma_k^{pq,rs}= 
\begin{cases}
\sigma_k^{pp,rr},\quad & p=q, r=s,p\neq r, \\ 
\sigma_k^{pq,qp}, & p=s,q=r,p\neq q, \\ 
0, & otherwise.
\end{cases}
\end{align*}

In particular, for all $p\neq q$, we have $-\sigma_k^{pq,qp}>0$ and 
\begin{align*}
-\sigma_k^{pq,qp}=\frac{\sigma_k^{pp}-\sigma_k^{qq}}{\lambda_q-\lambda_p},\quad\lambda_p\neq \lambda_q,
\end{align*}
where $\lambda_p$ and $\lambda_q$ are the eigenvalues of $D^2u$.
\end{lemm}

We also have the following basic properties of the Hessian quotient operator, see \cite{Lu24}

\begin{lemm}
\label{F^{ijkl}} Let $n\geq 2$, let $1\leq k<n$, let $F=\frac{\sigma_n}{\sigma_k}$ and let $u$ be a convex function. Suppose $D^2u$ is diagonalized at $x_0$. Then at $x_0$, we have 
\begin{align*}
F^{pq}=\left(\frac{\sigma_n^{pp}}{\sigma_k}-\frac{\sigma_n\sigma_k^{pp}}{\sigma_k^2}\right)\delta_{pq},
\end{align*}
\begin{align*}
F^{pq,rs}= 
\begin{cases}
-2\frac{\sigma_n^{pp}\sigma_k^{pp}}{\sigma_k^2}+2\frac{\sigma_n\left(\sigma_k^{pp}\right)^2}{\sigma_k^3},\quad & p=q=r=s, \\ 
\frac{\sigma_n^{pp,rr}}{\sigma_k}-\frac{\sigma_n^{pp}\sigma_k^{rr}}{\sigma_k^2}-\frac{\sigma_n^{rr}\sigma_k^{pp}}{\sigma_k^2}-\frac{\sigma_n\sigma_k^{pp,rr}}{\sigma_k^2}+2\frac{\sigma_n\sigma_k^{pp}\sigma_k^{rr}}{\sigma_k^3}, & p=q,r=s,p\neq r, \\ 
\frac{\sigma_n^{pq,qp}}{\sigma_k}-\frac{\sigma_n\sigma_k^{pq,qp}}{\sigma_k^2}, & p=s,q=r,p\neq q, \\ 
0, & otherwise.
\end{cases}
\end{align*}
In particular, for all $p\neq q$, we have $-F^{pq,qp}>0$ and 
\begin{align*}
-F^{pq,qp}=\frac{F^{pp}-F^{qq}}{\lambda_q-\lambda_p},\quad\lambda_p\neq\lambda_q,
\end{align*}
where $\lambda_p$ and $\lambda_q$ are the eigenvalues of $D^2u$.
\end{lemm}

\subsection{Legendre transform}

In this subsection, we will collect some basic formulas for Legendre transform.

Let $\Omega$ be a bounded convex domain in $\mathbb{R}^n$ and let $u(x)$ be a convex function in $\Omega$. Let $w(y)$ be the Legendre transform of the function $u(x)$. Then we have 
\begin{align}  \label{w-1}
(x,Du(x))=(Dw(y),y).
\end{align}

Note that $y(x)=Du(x)$ is a diffeomorphism. Let $\Omega^*$ be the domain of $y$, then $y$ maps $\partial \Omega$ to $\partial\Omega^*$.

By (\ref{w-1}), we have
\begin{align}  \label{w-2}
D^2w(y)=\frac{\partial x}{\partial y}=\left(\frac{\partial y}{\partial x}\right) ^{-1}=\left( D^2u(x)\right)^{-1}.
\end{align}

Let $n\geq 2$, $1\leq k<n$, let $F=\frac{\sigma_n}{\sigma_{k}}$ and let $G=\sigma_{n-k}$, then we have 
\begin{align}  \label{w-3}
G\left( D^2w(y)\right)=\sigma_{n-k}\left( \left( D^2u(x)\right)^{-1}\right)=\left(\frac{\sigma_n}{\sigma_{k}}\left( D^2u(x)\right)\right)^{-1}=\left(F\left( D^2u(x)\right)\right)^{-1}.
\end{align}

Suppose $D^2u$ is diagonalized at $x_0$. Then at $x_0$, by (\ref{w-2}) and (\ref{w-3}), we have 
\begin{align}  \label{g}
G^{ii}\left( D^2w\right)=F^{-2}F^{ii}\cdot w_{ii}^{-2}=F^{-2}F^{ii}u_{ii}^2.
\end{align}

Let $\varphi(x)$ be any function in $\Omega$. Denote $\varphi^*(y)=\varphi(x(y))$. Then at $x_0$, we have 
\begin{align*}
\frac{\partial \varphi}{\partial x_i}=\sum_j\frac{\partial y^j}{\partial x^i}\frac{\partial \varphi^*}{\partial y^j}=\sum_ju_{ij}\varphi^*_j,
\end{align*}
and 
\begin{align*}
\frac{\partial^2 \varphi}{\partial x_i^2}=\sum_j u_{iij}\varphi^*_j+\sum_{j,l}u_{ij}\varphi^*_{jl}u_{il}=\sum_j u_{iij}\varphi^*_j+u_{ii}^2 \varphi^*_{ii}.
\end{align*}

Together with (\ref{g}), at $x_0$, we have 
\begin{align}  \label{F^{ii}}
\sum_iF^{ii}\varphi_{ii}=&\ \sum_{i,j}F^{ii}u_{iij}\varphi^*_j+F^2\sum_iG^{ii}\varphi^*_{ii}, \\
\sum_iF^{ii}\varphi_i^2=&\ F^2\sum_iG^{ii}(\varphi^*_i)^2.  \notag
\end{align}

\section{Important Lemmas}

In this section, we will state several important lemmas, which are crucial to establish a priori estimates.

\subsection{Concavity inequality for Hessian operator}

In this subsection, we will state a concavity inequality for Hessian operator. The following lemma is a special case of Lemma 1.1 in \cite{RZhang} proved by Zhang. For the sake of completeness, we provide a self contained proof here. 

\begin{lemm}[Zhang \cite{RZhang}]\label{Concavity lemma-hessian} Let $n\geq 2$, $1< k<n$ and $\lambda_1\geq \cdots\geq \lambda_n>0$. Then for any $\xi\in \mathbb{R}^n$, there exists a positive constant $C$ depending only on $n,k$ and $\sigma_k$, such that if $\lambda_1>C$, then we have 
\begin{align*}
-\sum_{i\neq j}\sigma_k^{ii,jj}\xi_i\xi_j+\frac{2}{\sigma_k}\left(\sum_i\sigma_k^{ii}\xi_i\right)^2+\sum_{i\neq 1}\frac{\sigma_k^{ii}\xi_i^2}{\lambda_1}\geq \left(1+ \frac{1}{4nk}\right)\frac{\sigma_k^{11}\xi_1^2}{\lambda_1}.
\end{align*}
\end{lemm}

\begin{proof}
Without loss of generality, we may assume $\lambda_k\leq 1$. In fact if it is not true, then we have $\sigma_k\geq \lambda_1\cdots\lambda_k\geq
\lambda_1$. We may choose $C$ such that $\lambda_1> C>\sigma_k$ to obtain a contradiction.

Note that 
\begin{align*}
\sigma_k(|1)\leq c(n,k)\lambda_2\cdots\lambda_{k+1}\leq c(n,k)\sigma_k^{11}\lambda_{k+1}\leq c(n,k)\sigma_k^{11},
\end{align*}
we have used the fact that $\lambda_{k+1}\leq \lambda_k\leq 1$ in the last inequality. Here $c(n,k)$ denotes a positive constant depending only on $n$ and $k$, it may change from line to line.

Together with Lemma \ref{Sigma_k-Lemma-0}, we have 
\begin{align*}
\sigma_k=\lambda_1\sigma_k^{11}+\sigma_k(|1)\leq \left( 1+\frac{c(n,k)}{\lambda_1}\right) \lambda_1\sigma_k^{11}.
\end{align*}

Thus 
\begin{align}  \label{con-0}
\frac{\sigma_k^{11}}{\sigma_k}\geq \frac{1}{1+\frac{c(n,k)}{\lambda_1}}\cdot  \frac{1}{\lambda_1}.
\end{align}

We now start to prove the lemma. Recall that $\sigma_k^{\frac{1}{k}}$ is concave, and thus 
\begin{align*}
\frac{1}{k}\sigma_k^{\frac{1}{k}-1}\sigma_k^{ii,jj}+\frac{1}{k}\left(\frac{1}{k}-1\right)\sigma_k^{\frac{1}{k}-2}\sigma_k^{ii}\sigma_k^{jj}\leq 0,
\end{align*}
in the sense of comparison of matrices.

Contracting with $(0,\xi_2,\cdots,\xi_n)$, we have 
\begin{align*}
-\sum_{i\neq j:i,j\neq 1}\sigma_k^{ii,jj}\xi_i\xi_j\geq \left(\frac{1}{k}-1\right)\frac{1}{\sigma_k}\left(\sum_{i\neq 1}\sigma_k^{ii}\xi_i\right)^2.
\end{align*}

It follows that 
\begin{align}  \label{con-1}
&\ -\sum_{i\neq j}\sigma_k^{ii,jj}\xi_i\xi_j+\frac{1}{\sigma_k}\left(\sum_i\sigma_k^{ii}\xi_i\right)^2 \\
\geq&\ -2\sum_{i\neq 1} \sigma_k^{11,ii}\xi_1\xi_i+\frac{2}{\sigma_k}\sum_{i\neq 1}\sigma_k^{11}\sigma_k^{ii}\xi_1\xi_i+\frac{\left( \sigma_k^{11}\xi_1\right)^2}{\sigma_k}+\frac{1}{k\sigma_k}\left(\sum_{i\neq 1}\sigma_k^{ii}\xi_i\right)^2  \notag \\
=&\ \frac{2}{\sigma_k}\sum_{i\neq 1} \bigg( \sigma^2_{k-1}(|1i)-\sigma_{k}(|1i)\sigma_{k-2}(|1i)\bigg)\xi_1\xi_i+\frac{\left( \sigma_k^{11}\xi_1\right)^2}{\sigma_k}+\frac{1}{k\sigma_k}\left(\sum_{i\neq 1}\sigma_k^{ii}\xi_i\right)^2,  \notag
\end{align}
we have used the fact $\sigma_k^{11}\sigma_k^{ii}-\sigma_k\sigma_k^{11,ii}=\sigma^2_{k-1}(|1i)-\sigma_{k}(|1i)\sigma_{k-2}(|1i)$ in the last line, see (4.22) in \cite{GRW}.

Note that 
\begin{align*}
&\ \frac{\left( \sigma_k^{11}\xi_1\right)^2}{\sigma_k}+\frac{1}{k\sigma_k}\left(\sum_{i\neq 1}\sigma_k^{ii}\xi_i\right)^2+\frac{1}{\sigma_k}\left(\sum_i \sigma_k^{ii}\xi_i\right)^2 \\
=&\ \frac{2\left( \sigma_k^{11}\xi_1\right)^2}{\sigma_k}+ \frac{2\sigma_k^{11}\xi_1}{\sigma_k}\left(\sum_{i\neq 1}\sigma_k^{ii}\xi_i\right)+\frac{k+1}{k\sigma_k}\left(\sum_{i\neq 1}\sigma_k^{ii}\xi_i\right)^2  \notag \\
\geq &\ \left(2-\frac{k}{k+1}\right)\frac{\left( \sigma_k^{11}\xi_1\right)^2}{\sigma_k}\geq \left(1+\frac{1}{2k}\right)\frac{\left( \sigma_k^{11}\xi_1\right)^2}{\sigma_k}. 
\notag
\end{align*}

Plugging into (\ref{con-1}), we have 
\begin{align}  \label{con-3}
&\ -\sum_{i\neq j}\sigma_k^{ii,jj}\xi_i\xi_j+\frac{2}{\sigma_k}\left(\sum_i\sigma_k^{ii}\xi_i\right)^2 \\
\geq &\ \frac{2}{\sigma_k}\sum_{i\neq 1} \bigg( \sigma^2_{k-1}(|1i)-\sigma_{k}(|1i)\sigma_{k-2}(|1i)\bigg)\xi_1\xi_i+\left(1+\frac{1}{2k}\right)\frac{\left( \sigma_k^{11}\xi_1\right)^2}{\sigma_k}.  \notag
\end{align}

Recall that 
\begin{align}  \label{con-4}
\left( \sigma_k^{11}\right) ^2\geq (\lambda_2\cdots\lambda_k)^2.
\end{align}

By Newton-Maclaurin inequality, we have 
\begin{align*}
0\leq \sigma^2_{k-1}(|1i)-\sigma_{k}(|1i)\sigma_{k-2}(|1i)\leq \sigma^2_{k-1}(|1i).
\end{align*}

Thus for each $1<i\leq k$, we have 
\begin{align}  \label{con-4-1}
\left|\sigma^2_{k-1}(|1i)-\sigma_{k}(|1i)\sigma_{k-2}(|1i)\right|\leq &\ c(n,k) \left(\frac{\lambda_2\cdots\lambda_{k+1}}{\lambda_i}\right)^2, \\
\frac{\sigma_k\sigma_k^{ii}}{\lambda_1}\geq \frac{\lambda_1\cdots\lambda_k}{\lambda_1}\cdot \frac{\lambda_1\cdots\lambda_k}{\lambda_i}=&\ (\lambda_2\cdots\lambda_k)^2\frac{\lambda_1}{\lambda_i}.  \notag
\end{align}

Combining (\ref{con-4}) and (\ref{con-4-1}), we have 
\begin{align}  \label{con-5}
&\ \frac{2}{\sigma_k}\bigg( \sigma^2_{k-1}(|1i)-\sigma_{k}(|1i)\sigma_{k-2}(|1i)\bigg)\xi_1\xi_i+\frac{\left( \sigma_k^{11}\xi_1\right)^2}{2nk\sigma_k}+\frac{\sigma_k^{ii}\xi_i^2}{\lambda_1} \\
\geq&\ \frac{(\lambda_2\cdots\lambda_k)^2}{\sigma_k}\left(-c(n,k)\frac{\lambda_{k+1}^2}{\lambda_i^2} |\xi_1\xi_i|+\frac{\xi_1^2}{2nk}+\frac{\lambda_1}{\lambda_i}\xi_i^2\right)  \notag \\
\geq &\ \frac{(\lambda_2\cdots\lambda_k)^2}{\sigma_k}\left(\frac{\lambda_1}{\lambda_i}-c(n,k)\frac{\lambda_{k+1}^4}{\lambda_i^4}\right)\xi_i^2  \notag \\
\geq &\ \frac{(\lambda_2\cdots\lambda_k)^2}{\sigma_k}\left(\frac{\lambda_1}{\lambda_i}-c(n,k)\frac{\lambda_{k+1}}{\lambda_i}\right)\xi_i^2 \geq 0, 
\notag
\end{align}
as $\lambda_1$ is sufficiently large, $\lambda_i\geq \lambda_{k+1}$ and $\lambda_{k+1}\leq \lambda_k\leq 1$.

Similarly, for each $k+1\leq i\leq n$, we have 
\begin{align}  \label{con-4-2}
\left|\sigma^2_{k-1}(|1i)-\sigma_{k}(|1i)\sigma_{k-2}(|1i)\right|\leq&\ c(n,k) \left(\lambda_2\cdots\lambda_k\right)^2, \\
\frac{\sigma_k\sigma_k^{ii}}{\lambda_1}\geq \frac{\lambda_1\cdots\lambda_k}{\lambda_1}\cdot \lambda_1\cdots\lambda_{k-1}=&\ (\lambda_2\cdots\lambda_k)^2\frac{\lambda_1}{\lambda_k}.  \notag
\end{align}

Combining (\ref{con-4}) and (\ref{con-4-2}), we have 
\begin{align}  \label{con-6}
&\ \frac{2}{\sigma_k}\bigg( \sigma^2_{k-1}(|1i)-\sigma_{k}(|1i)\sigma_{k-2}(|1i)\bigg)\xi_1\xi_i+\frac{\left( \sigma_k^{11}\xi_1\right)^2}{2nk\sigma_k}+\frac{\sigma_k^{ii}\xi_i^2}{\lambda_1} \\
\geq &\ \frac{(\lambda_2\cdots\lambda_k)^2}{\sigma_k}\left(-c(n,k)|\xi_1\xi_i|+\frac{\xi_1^2}{2nk}+\frac{\lambda_1}{\lambda_k}\xi_i^2\right)  \notag
\\
\geq &\ \frac{(\lambda_2\cdots\lambda_k)^2}{\sigma_k}\left(\frac{\lambda_1}{\lambda_k}-c(n,k)\right)\xi_i^2\geq 0,  \notag
\end{align}
as $\lambda_1$ is sufficiently large and $\lambda_k\leq 1$.

Combining (\ref{con-3}), (\ref{con-5}) and (\ref{con-6}), together with (\ref{con-0}), we have 
\begin{align*}
&\ -\sum_{i\neq j}\sigma_k^{ii,jj}\xi_i\xi_j+\frac{2}{\sigma_k}\left(\sum_i\sigma_k^{ii}\xi_i\right)^2+\sum_{i\neq 1}\frac{\sigma_k^{ii}\xi_i^2}{\lambda_1} \\
\geq &\ \left(1+\frac{1}{2nk}\right)\frac{\left( \sigma_k^{11}\xi_1\right)^2}{\sigma_k}\geq \left(1+\frac{1}{2nk}\right) \frac{1}{1+\frac{c(n,k)}{\lambda_1}}\frac{\sigma_k^{11}\xi_1^2}{\lambda_1} \\
\geq&\ \left(1+\frac{1}{4nk}\right)\frac{\sigma_k^{11}\xi_1^2}{\lambda_1},
\end{align*}
as $\lambda_1$ is sufficiently large.

The lemma is now proved.
\end{proof}

\subsection{Concavity inequality for Hessian quotient operator}

In this subsection, we will state a concavity inequality for Hessian quotient operator. The following lemma was proved by Guan and Sroka \cite{GS}, see also Proposition 2.5 in \cite{Sro}.
\begin{lemm}[Guan and Sroka \cite{GS}]
\label{Concavity lemma} Let $n\geq 2$, $1\leq k<n$ and $F=\frac{\sigma_n}{\sigma_k}$ with $\lambda_1\geq \cdots\geq \lambda_n>0$. Then for any $\xi\in 
\mathbb{R}^n$, we have 
\begin{align*}
-\sum_{i,j}F^{ii,jj}\xi_i\xi_j +\frac{1}{F}\left(\sum_i
F^{ii}\xi_i\right)^2\geq \left( 1+c(n,k)\right)\frac{F^{11}\xi_1^2}{\lambda_1},
\end{align*}
where $c(n,k)$ is a positive constant depending only on $n$ and $k$.
\end{lemm}

\subsection{Jabobi inequality}

In this subsection, we will prove a Jacobi inequality for the Hessian quotient equation.

\begin{lemm}
\label{Jacobi} Let $n\geq 2$, $1\leq k<n$ and let $\Omega$ be a bounded domain. Let $f\in C^{1,1}\left(\overline{\Omega}\times \mathbb{R} \right)$ be a positive function and let $u\in C^4\left(\overline{\Omega}\right) $ be a convex solution of (\ref{main-eqn}). For any $x_0\in \Omega$, suppose that $D^2u$ is diagonalized at $x_0$ such that $\lambda_1\geq \cdots\geq \lambda_n$. Let $b=\ln\lambda_1$, then at $x_0$, we have 
\begin{align*}
\sum_i F^{ii}b_{ii}\geq c(n,k)\sum_i F^{ii}b_i^2-C,
\end{align*}
in the viscosity sense, where $c(n,k)$ is a positive constant depending only on $n$ and $k$, and $C$ is a positive constant depending only on $n$, $\|u\|_{C^{0,1 }\left(\overline{\Omega}\right) }$, $\min_{\overline{\Omega}\times [-M,M]}f $ and $\|f\|_{C^{1,1}\left( \overline{\Omega}\times [-M,M]\right)}$. Here $M$ is a large constant satisfying $\|u\|_{L^\infty\left( \overline{\Omega}\right) }\leq M$.
\end{lemm}

\begin{proof}
The proof is almost the same as the proof of Lemma 4.1 in \cite{Lu24}. The only difference is that instead of using Lemma 3.1 and Lemma 3.2 in \cite{Lu24}, we use Lemma \ref{Concavity lemma} here. The rest are exactly the same, we omit the detail here.
\end{proof}

\subsection{Comparison principle}

In this subsection, we will state a classic comparison principle for viscosity solution, see for instance in \cite{IL}. Since our equation is quite simple, we will give a straightforward proof here.
\begin{lemm}\label{compare 1} 
Let $n\geq 2$, $1\leq k<n$ and let $\Omega$ be a bounded domain. Let $f\in C^0\left(\Omega\right) $ be a positive function.

$(1)$ Let $u\in C^0\left( \overline{\Omega}\right) $ be a viscosity subsolution of (\ref{viscosity}), and let $v\in C^0\left( \overline{\Omega}\right) \cap C^2\left( \Omega\right)$ be a convex supersolution of (\ref{viscosity}). If $u\leq v$ on $\partial \Omega$, then $u\leq v$ in $\Omega$.

$(2)$ Let $u\in C^0\left( \overline{\Omega}\right) $ be a viscosity supersolution of (\ref{viscosity}), and let $v\in C^0\left( \overline{\Omega}\right) \cap C^2\left( \Omega\right)$ be a convex subsolution of (\ref{viscosity}). If $u\geq v$ on $\partial \Omega$, then $u\geq v$ in $\Omega$.

\end{lemm}

\begin{proof}
We only need to prove $(1)$, $(2)$ can be proved in a similar manner. 

For the sake of simplicity, define
\begin{align*}
\tilde{F}=\left(\frac{\sigma_n}{\sigma_k}\right)^{\frac{1}{n-k}},\quad \tilde{f}=f^{\frac{1}{n-k}}.
\end{align*}

Then $u$ is a viscosity subsolution of $\tilde{F} \left( D^2w\right)=\tilde{f}$ and $v$ is a convex supersolution of $\tilde{F} \left( D^2w\right)=\tilde{f}$.

Let
\begin{equation*}
u_{\epsilon}=u+\epsilon |x|^{2}.
\end{equation*}

{\bf Claim}: $u_\epsilon$ is a viscosity subsolution of $\tilde{F} \left( D^2w\right)=\tilde{f}+c(n,k)\epsilon$, where $c(n,k)$ is a positive constant depending only on $n$ and $k$. 

\medskip

For any $x_0\in \Omega$, let $\varphi _{\epsilon}\in C^2(\Omega)$ be a function satisfying
\begin{align*}
\varphi_\epsilon(x_0)=u_\epsilon(x_0),\quad and\quad \varphi_\epsilon\geq u_\epsilon,\quad in\quad \Omega.
\end{align*}

Let $\varphi =\varphi _{\epsilon }-\epsilon |x|^2 $. Then we have
\begin{align*}
\varphi(x_0)=u(x_0),\quad and\quad \varphi\geq u,\quad in\quad \Omega.
\end{align*}

Since $u$ is a viscosity subsolution of $\tilde{F} \left( D^2w \right)=\tilde{f}$, we have
\begin{equation*}
\tilde{F} \left( D^2\varphi(x_0)\right) \geq \tilde{f}(x_0).
\end{equation*}

It is not hard to verify that $D^{2}\varphi (x_{0})\geq 0$ since $u$ is convex. Let $A=D^{2}\varphi (x_{0})$ and $B=2\epsilon I$. Since $\tilde{F}$ is concave on nonnegative matrices, we have $2\tilde{F}\left( \frac{A+B}{2}\right) \geq \tilde{F}( A) +\tilde{F}( B) $. Together with the fact that $\tilde{F}$ is homogeneous of degree one, we have
\begin{equation*}
\tilde{F}\left( D^2\varphi _\epsilon (x_0)\right) \geq \ \tilde{F}\left( D^2\varphi (x_0)\right) +\ \tilde{F}( 2\epsilon I)
\geq \tilde{f}(x_{0})+c(n,k)\epsilon,
\end{equation*}
where $c(n,k)=2\tilde{F}(I)$ is a positive constant depending only on $n$ and $k$.

The claim is now proved.

\medskip

We are now in position to prove the lemma. Suppose not, then we have 
\begin{equation*}
\max_{\overline{\Omega}}( u-v) >\max_{\partial \Omega}(u-v).
\end{equation*}

In particular, there exists $\epsilon >0$ such that 
\begin{equation*}
\max_{\overline{\Omega}}\left( u_{\epsilon }-v\right) >\max_{\partial \Omega}\left( u_{\epsilon }-v\right).
\end{equation*}

Therefore, there exists $x_0\in \Omega$ such that 
\begin{align*}
u_\epsilon(x_0)-v(x_0)=M=\max_{\overline{\Omega}}\left( u_{\epsilon }-v\right).
\end{align*}

Thus
\begin{align*}
v(x_0)+M=u_\epsilon(x_0),\quad and\quad v+M\geq u_\epsilon,\quad in\quad \Omega.
\end{align*}

By Claim, we have 
\begin{equation*}
\tilde{F}\left( D^2v(x_0)\right)\geq \tilde{f}(x_0)+c(n,k)\epsilon.
\end{equation*}

This contradicts the fact that $v$ is a supersolution. The lemma is now proved.
\end{proof}

\section{Strict positivity of the Hessian}

In this section, we will show that the Hessian is uniformly bounded below by a positive constant. This is the key ingredient in the proof of our main theorem. 
\begin{lemm}\label{Pogolev-1} 
Let $n\geq 2$, $1\leq k<n$, let $\Omega$ be a smooth bounded convex domain. Let $f\in C^{1,1}\left(\overline{\Omega}\times\mathbb{R}\times\mathbb{R}^n\right)$ be a positive function and let $u\in C^4\left(\overline{\Omega}\right) $ be a convex solution of the following Dirichlet problem 
\begin{align*}
\begin{cases}
F\left( D^2u \right)=f(x,u,Du),\quad & in\quad\Omega, \\ 
u=0, & on\quad\partial\Omega.
\end{cases}
\end{align*}
where $F=\frac{\sigma_n}{\sigma_{n-k}}$. Let $\lambda_{min}$ be the smallest eigenvalue of $D^2u$, then we have
\begin{align*}
\frac{(-u)^\beta}{\lambda_{min}}\leq C,
\end{align*}
where $\beta$ and $C$ are positive constants depending only on $n$, $k$, diameter of $\Omega$, $\|u\|_{C^{0,1}\left(\overline{\Omega}\right) }$, $\min_{\overline{\Omega}\times[-M,M]\times\overline{B}_M}f$ and $\|f\|_{C^{1,1}\left(\overline{\Omega}\times[-M,M]\times\overline{B}_M\right)}$. Here $M$ is a universal constant satisfying $\|u\|_{C^{0,1} \left(\overline{\Omega}\right) }\leq M$. 
\end{lemm}

\begin{proof}
Let $w(y)$ be the Legendre transform of $u(x)$. Then we have
\begin{align}  \label{www}
w(y)=\sup_{x\in \Omega}\left(x\cdot y-u(x)\right)=Dw(y)\cdot y-u(Dw(y)).
\end{align}

Denote $u^*(y)=u(x(y))$ and define 
\begin{align*}
f^*(Dw,u^*,y)=\left( f(x,u,Du)\right) ^{-1}.
\end{align*}

By (\ref{w-1}) and (\ref{w-3}), $w$ is a convex solution of the following boundary value problem 
\begin{align}  \label{equation-w}
\begin{cases}
\sigma_k\left( D^2w\right)=f^*(Dw,u^*,y),\quad & in\quad\Omega^*, \\ 
u^*=0, & on\quad\partial\Omega^*.
\end{cases}
\end{align}

Let $\mu_1$ be the largest eigenvalue of $D^2w$. By (\ref{w-2}), to prove the lemma, we only need to bound $(-u^*)^\beta\mu_1$.

Since the diameter of $\Omega$ and $\|u\|_{C^{0,1}\left( \overline{\Omega}\right) }$ is bounded, by (\ref{w-1}), the diameter of $\Omega^*$ and $|Dw|$ is bounded. By (\ref{www}), we conclude that $|w|$ is bounded. It follows that $\|w\|_{C^{0,1}\left( \overline{\Omega}^*\right) }$ is bounded.

\medskip

If $k=1$, by (\ref{equation-w}) and the fact that $w$ is convex, we conclude that $D^2w$ is bounded. Therefore, we only need to prove the case $1<k<n$. Consider the test function 
\begin{align*}
Q=(-u^*)^\beta \mu_1 e^{\frac{\alpha}{2} |\nabla w|^2},
\end{align*}
where $\alpha$ and $\beta$ are large constants to be determined later.

Since $u$ is convex and $u=0$ on $\partial\Omega$, we conclude that $u<0$ in $\Omega$. Consequently, $-u^*>0$ in $\Omega^*$. Therefore, $Q$ will achieve maximum at an interior point of $\Omega^*$, say $y_0$. At $y_0$, we can choose an orthonormal frame such that $D^2w$ is diagonalized. Without loss of generality, we may assume $\mu_1$ has multiplicity $m$, i.e. 
\begin{align*}
\mu_1=\cdots=\mu_m>\mu_{m+1}\geq \cdots\geq \mu_n.
\end{align*}

In the following, we will use $C$ to denote a universal constant depending only on $n$, $k$, diameter of $\Omega$, $\|u\|_{C^{0,1}\left( \overline{\Omega}\right) }$, $\min_{\overline{\Omega}\times[-M,M]\times\overline{B}_M}f$ and $\|f\|_{C^{1,1}\left(\overline{\Omega}\times[-M,M]\right)\times\overline{B}_M}$. It may change from line to line.

By Lemma 5 in \cite{BCD}, at $y_0$, we have 
\begin{align}  \label{BCDDDD}
\delta_{ab}\cdot (\mu_1)_i=w_{ab i},\quad 1\leq a,b\leq m,
\end{align}
\begin{align*}
(\mu_1)_{ii}\geq w_{11ii}+2\sum_{p>m}\frac{w_{1pi}^2}{\mu_1-\mu_p},
\end{align*}
in the viscosity sense.

At $y_0$, we have 
\begin{align}  \label{General-critical}
0=\beta\frac{u^*_i}{u^*}+\frac{(\mu_1)_i}{\mu_1} +\alpha\sum_j
w_jw_{ji},
\end{align}
\begin{align}  \label{General-1}
0\geq \beta\frac{u^*_{ii}}{u^*}-\beta\frac{(u^*_i)^2}{(u^*)^2}+\frac{(\mu_1)_{ii}}{\mu_1}-\frac{(\mu_1)_i^2}{\mu_1^2}+\alpha\sum_j\left( w_{ji}^2+ w_jw_{jii}\right),
\end{align}
in the viscosity sense.

Plugging (\ref{BCDDDD}) into (\ref{General-1}), we have 
\begin{align*}
0\geq&\ \beta\frac{u^*_{ii}}{u^*}-\beta\frac{(u^*_i)^2}{(u^*)^2} +\frac{w_{11ii}}{\mu_1}+2\sum_{p>m}\frac{w_{1pi}^2}{\mu_1(\mu_1-\mu_p)}-\frac{ w_{11i}^2}{\mu_1^2}+\alpha\sum_j\left( w_{ji}^2+w_jw_{jii}\right).
\end{align*}

Contracting with $\sigma_k^{ii}$, we have 
\begin{align}  \label{General-2}
0\geq &\ \beta\sum_i\frac{\sigma_k^{ii} u^*_{ii}}{u^*}-\beta\sum_i\frac{\sigma_k^{ii}(u^*_i)^2}{(u^*)^2} +\sum_i\frac{\sigma_k^{ii}w_{ii11}}{\mu_1}+2\sum_i\sum_{p>m}\frac{\sigma_k^{ii}w_{1pi}^2}{\mu_1(\mu_1-\mu_p)} \\
&\ -\sum_i\frac{ \sigma_k^{ii}w_{11i}^2}{\mu_1^2}+\alpha\sum_i\sigma_k^{ii}w_{ii}^2+\alpha\sum_{j,i} \sigma_k^{ii}w_jw_{jii}. \notag
\end{align}

Differentiating equation (\ref{equation-w}), we have 
\begin{align*}
\alpha\sum_{j,i} \sigma_k^{ii}w_jw_{jii}=\alpha\sum_j w_jf^*_j=\alpha\sum_jw_j\left(\sum_l d_{p_l} f^*\cdot w_{jl}+ d_{u^*}f^*\cdot u^*_j+d_{y_j}f^*\right).
\end{align*}

Plugging into (\ref{General-2}), we have 
\begin{align}  \label{General-3}
0\geq &\ \beta\sum_i\frac{\sigma_k^{ii} u^*_{ii}}{u^*}-\beta\sum_i\frac{\sigma_k^{ii}(u^*_i)^2}{(u^*)^2} +\sum_i\frac{\sigma_k^{ii}w_{ii11}}{\mu_1}+2\sum_i\sum_{p>m}\frac{\sigma_k^{ii}w_{1pi}^2}{\mu_1(\mu_1-\mu_p)} \\
&\ -\sum_i\frac{ \sigma_k^{ii}w_{11i}^2}{\mu_1^2}+\alpha\sum_i\sigma_k^{ii}w_{ii}^2+\alpha \sum_jw_j\left(d_{p_j} f^*\cdot w_{jj}+d_{u^*}f^*\cdot u^*_j\right)-C\alpha. \notag
\end{align}

Differentiating equation (\ref{equation-w}) twice, we have 
\begin{align}  \label{General-3-1}
&\ \sum_i\sigma_k^{ii}w_{ii11}+\sum_{p,q,r,s}\sigma_k^{pq,rs}w_{pq1}w_{rs1}=f^*_{11} \\
=&\ \left(\sum_j d_{p_j} f^*\cdot w_{j1}+ d_{u^*}f^*\cdot u^*_1+d_{y_1}f^*\right)_1 \notag  \\
=&\ \sum_{j,l}d_{p_jp_l} f^*\cdot w_{j1}w_{l1}+2\sum_j d_{p_ju^*} f^*\cdot w_{j1}u^*_1+2\sum_j d_{p_jy_1}f^*\cdot w_{j1} \notag \\
&\ +\sum_j d_{p_j} f^*\cdot w_{j11}+d_{u^*u^*}f^*\cdot (u^*_1)^2+2d_{u^*y_1}f^*\cdot u^*_1 \notag \\
&\ + d_{u^*}f^*\cdot u^*_{11}+d_{y_1y_1}f^* \notag \\ 
\geq &\  2d_{p_1u^*} f^*\cdot w_{11}u^*_1+\sum_j d_{p_j} f^*\cdot w_{j11}+d_{u^*u^*}f^*\cdot (u^*_1)^2  \notag \\
&\ +2d_{u^*y_1}f^*\cdot u^*_1+ d_{u^*}f^*\cdot u^*_{11}-C\mu_1^2-C.\notag
\end{align}

By (\ref{w-1}), we have 
\begin{align}  \label{u^*}
u^*_i=\frac{\partial u(x(y))}{\partial y_i}=\sum_j u_j \frac{\partial x_j}{\partial y_i}=\sum_j u_j w_{ji}.
\end{align}

Similarly, 
\begin{align}  \label{u^**}
u^*_{ii}=\sum_{j,l} u_{jl} w_{ji}w_{li}+\sum_j u_jw_{jii}.
\end{align}

Together with (\ref{w-2}), we have 
\begin{align*}
u^*_{11}=u_{11}w_{11}^2+\sum_j u_jw_{j11}=w_{11}+\sum_j u_jw_{j11}.
\end{align*}

Plugging the above equations into (\ref{General-3-1}), we have 
\begin{align*}
\sum_i\sigma_k^{ii}w_{ii11}+\sum_{p,q,r,s}\sigma_k^{pq,rs}w_{pq1}w_{rs1}\geq \sum_j d_{p_j} f^*\cdot w_{j11}+d_{u^*}f^*\sum_j u_jw_{j11}-C\mu_1^2-C.
\end{align*}

Plugging into (\ref{General-3}), we have 
\begin{align}  \label{General-4}
0\geq &\ \beta\sum_i\frac{\sigma_k^{ii} u^*_{ii}}{u^*}-\beta\sum_i\frac{\sigma_k^{ii}(u^*_i)^2}{(u^*)^2} -\sum_{p,q,r,s}\frac{\sigma_k^{pq,rs}w_{pq1}w_{rs1}}{\mu_1}  \\
&\ +2\sum_i\sum_{p>m}\frac{\sigma_k^{ii}w_{1pi}^2}{\mu_1(\mu_1-\mu_p)}-\sum_i\frac{ \sigma_k^{ii}w_{11i}^2}{\mu_1^2}+\alpha \sum_i\sigma_k^{ii}w_{ii}^2  \notag \\
&\ +\alpha \sum_j\left(d_{p_j} f^*\cdot w_jw_{jj}+d_{u^*}f^*\cdot w_ju^*_j\right) \notag \\
&\ +\sum_j \left( d_{p_j} f^*+d_{u^*}f^*\cdot u_j\right)\frac{w_{j11}}{\mu_1}-C\mu_1-C\alpha .  \notag
\end{align}

By (\ref{u^*}), we have 
\begin{align*}
&\ \alpha \sum_j\left(d_{p_j} f^*\cdot w_jw_{jj}+ d_{u^*}f^*\cdot w_ju^*_j\right) \\
=&\ \alpha \sum_j\left(d_{p_j} f^*\cdot w_jw_{jj}+ d_{u^*}f^*\cdot w_ju_jw_{jj}\right) \\
=&\ \alpha\sum_j \left( d_{p_j} f^*+ d_{u^*}f^*\cdot u_j\right) w_jw_{jj}.
\end{align*}

Plugging into (\ref{General-4}), we have 
\begin{align}  \label{General-4-1}
0\geq &\ \beta\sum_i\frac{\sigma_k^{ii} u^*_{ii}}{u^*}-\beta\sum_i\frac{\sigma_k^{ii}(u^*_i)^2}{(u^*)^2} -\sum_{p,q,r,s}\frac{\sigma_k^{pq,rs}w_{pq1}w_{rs1}}{\mu_1}  \\
&\ +2\sum_i\sum_{p>m}\frac{\sigma_k^{ii}w_{1pi}^2}{\mu_1(\mu_1-\mu_p)}-\sum_i\frac{\sigma_k^{ii}w_{11i}^2}{\mu_1^2}+\alpha\sum_i\sigma_k^{ii}w_{ii}^2 \notag \\
&\ +\sum_j \left( d_{p_j}f^*+d_{u^*}f^*\cdot u_j\right)\left(\frac{w_{j11}}{\mu_1}+\alpha w_jw_{jj}\right) -C\mu_1-C\alpha . \notag
\end{align}

By (\ref{u^*}), (\ref{u^**}), (\ref{w-2}), Lemma \ref{Sigma_k-Lemma-0} and equation \ref{equation-w}, we have 
\begin{align*}
\beta\sum_i\frac{\sigma_k^{ii} u^*_{ii}}{u^*}=&\ \beta\sum_i\frac{\sigma_k^{ii}}{u^*}\left(u_{ii} w_{ii}^2+\sum_j u_jw_{jii}\right)=\beta\sum_i\frac{\sigma_k^{ii}w_{ii}}{u^*}+\beta\sum_j \frac{u_jf^*_j}{u^*} \\
= &\ \frac{\beta kf^*}{u^*}+ \beta\sum_j\frac{u_j}{u^*}\left(\sum_l d_{p_l}f^*\cdot w_{lj}+d_{u^*}f^*\cdot u^*_j+f^*_{y_j}\right) \\
\geq &\ -\mu_1+\beta\sum_j \frac{u^*_j}{u^*}\left( d_{p_j} f^*+d_{u^*}f^*\cdot u_j\right).
\end{align*}
We have used the fact that $(-u^*)^\beta \mu_1$ is sufficiently large in last line.

Plugging into (\ref{General-4-1}), together with critical equation \ref{General-critical}, we have 
\begin{align}  \label{General-4-2}
0\geq &\ -\beta\sum_i\frac{\sigma_k^{ii}(u^*_i)^2}{(u^*)^2} -\sum_{p,q,r,s}\frac{\sigma_k^{pq,rs}w_{pq1}w_{rs1}}{\mu_1} +2\sum_i\sum_{p>m}\frac{\sigma_k^{ii}w_{1pi}^2}{\mu_1(\mu_1-\mu_p)} \\
&\ -\sum_i\frac{ \sigma_k^{ii}w_{11i}^2}{\mu_1^2}+\alpha\sum_i\sigma_k^{ii}w_{ii}^2-C\mu_1-C\alpha.  \notag
\end{align}

By Lemma \ref{F^ijkl-1}, we have 
\begin{align*}
-\sum_{p,q,r,s} \sigma_k^{pq,rs}w_{pq1}w_{rs1}=-\sum_{p\neq q}\sigma_k^{pp,qq}w_{pp1}w_{qq1}-\sum_{p\neq q}\sigma_k^{pq,qp}w_{pq1}^2.
\end{align*}

Plugging into (\ref{General-4-2}), we have 
\begin{align}  \label{General-5}
0\geq &\ -\beta\sum_i\frac{\sigma_k^{ii}(u^*_i)^2}{(u^*)^2} -\sum_{p\neq q}\frac{\sigma_k^{pp,qq}w_{pp1}w_{qq1}}{\lambda_1}-\sum_{p\neq q} \frac{\sigma_k^{pq,qp}w_{pq1}^2}{\mu_1} \\
&\ +2\sum_i\sum_{p>m}\frac{\sigma_k^{ii}w_{1pi}^2}{\mu_1(\mu_1-\mu_p)}-\sum_i\frac{ \sigma_k^{ii}w_{11i}^2}{\mu_1^2}+\alpha \sum_i\sigma_k^{ii}w_{ii}^2-C\mu_1-C\alpha.  \notag
\end{align}

By Lemma \ref{F^ijkl-1}, we have 
\begin{align}  \label{General-6}
-\sum_{p\neq q} \frac{\sigma_k^{pq,qp}w_{pq1}^2}{\mu_1}\geq -2\sum_{i>m}\frac{\sigma_k^{1i,i1}w_{11i}^2}{\mu_1}= 2\sum_{i>m}\frac{(\sigma_k^{ii}-\sigma_k^{11})w_{11i}^2}{\mu_1(\mu_1-\mu_i)}.
\end{align}

On the other hand, 
\begin{align}  \label{General-7}
2\sum_i\sum_{p>m}\frac{\sigma_k^{ii}w_{1pi}^2}{\mu_1(\mu_1-\mu_p)}\geq&\ 2\sum_{p>m}\frac{\sigma_k^{pp}w_{1pp}^2}{\mu_1(\mu_1-\mu_p)}+2\sum_{p>m}\frac{\sigma_k^{11}w_{1p1}^2}{\mu_1(\mu_1-\mu_p)} \\
=&\ 2\sum_{i>m}\frac{\sigma_k^{ii}w_{ii1}^2}{\mu_1(\mu_1-\mu_i)}+2\sum_{i>m}\frac{\sigma_k^{11}w_{11i}^2}{\mu_1(\mu_1-\mu_i)}. 
\notag
\end{align}

By (\ref{BCDDDD}), we have 
\begin{align}  \label{General-8}
w_{11i}=w_{1i1}=\delta_{li}\cdot (\mu_1)_1=0,\quad 1<i\leq m.
\end{align}

Plugging (\ref{General-6}), (\ref{General-7}) and (\ref{General-8}) into (\ref{General-5}), we have 
\begin{align}  \label{General-9}
0\geq &\ -\beta\sum_i\frac{\sigma_k^{ii}(u^*_i)^2}{(u^*)^2} -\sum_{p\neq q}\frac{\sigma_k^{pp,qq}w_{pp1}w_{qq1}}{\mu_1}+ 2\sum_{i>m}\frac{\sigma_k^{ii}w_{11i}^2}{\mu_1(\mu_1-\mu_i)} \\
&\ +2\sum_{i>m}\frac{\sigma_k^{ii}w_{ii1}^2}{\mu_1(\mu_1-\mu_i)}-\sum_i\frac{ \sigma_k^{ii}w_{11i}^2}{\mu_1^2}+\alpha \sum_i\sigma_k^{ii}w_{ii}^2-C\mu_1-C\alpha  \notag \\
\geq &\ -\beta\sum_i\frac{\sigma_k^{ii}(u^*_i)^2}{(u^*)^2} -\sum_{p\neq q}\frac{\sigma_k^{pp,qq}w_{pp1}w_{qq1}}{\mu_1}+\sum_{i>m}\frac{\sigma_k^{ii}w_{11i}^2}{\mu_1^2}+\sum_{i>m}\frac{\sigma_k^{ii}w_{ii1}^2}{\mu_1^2}  \notag \\
&\ -\frac{ \sigma_k^{11}w_{111}^2}{\mu_1^2} +\alpha \sum_i\sigma_k^{ii}w_{ii}^2-C\mu_1-C\alpha.  \notag
\end{align}

By (\ref{BCDDDD}), we have 
\begin{align*}
w_{ii1}=w_{1ii}=\delta_{li}\cdot (\mu_1)_i=0,\quad 1<i\leq m.
\end{align*}

Together with Lemma \ref{Concavity lemma-hessian} and (\ref{u^*}), we have 
\begin{align*}
&\ -\sum_{p\neq q}\frac{\sigma_k^{pp,qq}w_{pp1}w_{qq1}}{\mu_1}+\sum_{i>m}\frac{\sigma_k^{ii}w_{ii1}^2}{\mu_1^2}- \frac{ \sigma_k^{11}w_{111}^2}{\mu_1^2} \\
\geq &\ -2\frac{(f^*_1)^2}{f^*\mu_1}+ c(n,k)\frac{ \sigma_k^{11}w_{111}^2}{\mu_1^2}\geq c(n,k)\frac{ \sigma_k^{11}w_{111}^2}{\mu_1^2}-C\mu_1,
\end{align*}
where $c(n,k)$ is a positive constant depending only on $n$ and $k$.

Plugging into (\ref{General-9}), together with (\ref{General-8}), we have 
\begin{align}  \label{General-10}
0\geq &\ -\beta\sum_i\frac{\sigma_k^{ii}(u^*_i)^2}{(u^*)^2} +c(n,k)\sum_i\frac{ \sigma_k^{ii}w_{11i}^2}{\mu_1^2}+\alpha \sum_i\sigma_k^{ii}w_{ii}^2-C\mu_1-C\alpha.
\end{align}

By critical equation (\ref{General-critical}), we have 
\begin{align*}
\beta^2\frac{(u^*_i)^2}{(u^*)^2} =\left(\frac{w_{11i}}{\mu_1}+\alpha w_iw_{ii}\right)^2\leq 2\frac{w_{11i}^2}{\mu_1^2}+2\alpha^2w_i^2w_{ii}^2.
\end{align*}

Plugging into (\ref{General-10}), we have 
\begin{align}  \label{General-11}
0\geq &\ \left( c(n,k)-\frac{2}{\beta}\right)\sum_i\frac{\sigma_k^{ii}w_{11i}^2}{\mu_1^2}+\left(\alpha-C\frac{\alpha^2}{\beta}\right)\sum_i \sigma_k^{ii}w_{ii}^2-C\mu_1-C\alpha \\
\geq &\ \left(\alpha-C\right) \sum_i \sigma_k^{ii}w_{ii}^2-C\mu_1-C\alpha,  \notag
\end{align}
by choosing $\beta=\alpha^2$ sufficiently large.

By Lemma \ref{Sigma_k-Lemma-0} and Newton-Maclaurin inequality, we have 
\begin{align*}
\sum_i\sigma_k^{ii}w_{ii}^2=\sigma_1\sigma_k-(k+1)\sigma_{k+1}\geq c(n,k)\sigma_1\sigma_k\geq C\mu_1.
\end{align*}

Plugging into (\ref{General-11}), we conclude that 
\begin{align*}
0\geq C(\alpha-C)\mu_1-C\mu_1-C\alpha.
\end{align*}

It follows that $\mu_1\leq C$ by choosing $\alpha$ sufficiently large. The lemma is now proved.
\end{proof}

\section{Legendre transform}

In this section, we will follow the steps in \cite{Lu23,Lu24} (inspired by Shankar and Yuan \cite{SY-20}) to bound $b=\ln\lambda_1+K_0$ by its integral using Legendre transform. Here $K_0$ is a positive constant such that $b$ is always positive. We note that $K_0$ only depends on $\min{f}$, $n$ and $k$.

We first show that $b$ is a subsolution of a uniformly elliptic equation in the new coordinate system. 
\begin{lemm}
\label{Jacobi-g} Let $n\geq 2$, $1\leq k<n$. Let $f\in C^{1,1}\left( \overline{\Sigma}_2\times \mathbb{R}\right) $ be a positive function and let $u\in C^4\left( \overline{\Sigma}_2\right)$ be a convex solution of (\ref{main-eqn}). For any $x_0\in \Sigma_1$, suppose that $D^2u$ is diagonalized at $x_0$ such that $\lambda_1\geq \cdots\geq \lambda_n$ and $-K_0$ is a uniform lower bound of $\ln\lambda_1$. Let $b=\ln\lambda_1+K_0$, then at $y(x_0)$, we have 
\begin{align*}
\sum_i G^{ii} b^*_{ii}\geq -C,
\end{align*}
in the viscosity sense, where $G=\sigma_{n-k}$ and $C$ is a positive constant depending only on $n$, $k$, diameter of $\Sigma_2$, $\|u\|_{C^{0,1}\left( \overline{\Sigma}_2\right) }$, $\min_{\overline{\Sigma}_2\times [0,2]}f $ and $\|f\|_{C^{1,1}\left( \overline{\Sigma}_2\times [0,2]\right) }$. We use the notation $b^*(y)=b(x(y))$.
\end{lemm}

\begin{proof}
By Lemma \ref{Jacobi}, we have 
\begin{align*}
\sum_iF^{ii}b_{ii}\geq c(n,k)\sum_iF^{ii}b_i^2-C,
\end{align*}
in the viscosity sense, where $C$ is a positive constant depending only on $n$, $k$, diameter of $\Sigma_2$, $\|u\|_{C^{0,1}\left( \overline{\Sigma}_2\right) }$,  $\min_{\overline{\Sigma}_2\times [0,2]}f $ and $\|f\|_{C^{1,1}\left( \overline{\Sigma}_2\times [0,2]\right)}$. 

Together with (\ref{g}) and (\ref{F^{ii}}), we have 
\begin{align*}
\sum_jf_jb^*_j+f^2\sum_iG^{ii} b^*_{ii}\geq c(n,k)f^2\sum_iG^{ii}(b^*_i)^2-C.
\end{align*}

It follows that 
\begin{align*}
\sum_iG^{ii}b^*_{ii}\geq&\ c(n,k)\sum_iG^{ii}(b^*_i)^2-f^{-2}\sum_if_ib^*_i-C \\
\geq&\ c(n,k)\sum_iG^{ii}(b^*_i)^2-C\sum_i|b^*_i|-C  \notag \\
\geq&\ -C\sum_i \frac{1}{G^{ii}}-C.  \notag
\end{align*}

By Lemma \ref{Pogolev-1}, $|D^2w|$ is bounded in $(\Sigma_1)^*$. It follows that 
\begin{align}  \label{g-1}
\frac{1}{C}\leq G^{ii}\leq C,\quad 1\leq i\leq n.
\end{align}

Consequently, 
\begin{align}  \label{g-2}
\sum_iG^{ii}b^*_{ii}\geq -C.
\end{align}

The lemma is now proved.
\end{proof}

We now state the mean value inequality.

\begin{lemm}
\label{mean} Let $n\geq 2$, $1\leq k<n$. Let $f\in C^{1,1}\left( \overline{\Sigma}_2\times \mathbb{R}\right) $ be a positive function and let $u\in C^4\left(\overline{\Sigma}_2\right)$ be a convex solution of (\ref{main-eqn}). Let $b=\ln\lambda_1+K_0$, where $\lambda_1$ is the largest eigenvalue of $D^2u$ and $-K_0$ is a uniform lower bound of $\ln\lambda_1$. Then there exist positive constants $\epsilon$ and $C$ depending only on $n$, $k$, diameter of $\Sigma_2$, $\|u\|_{C^{0,1}\left( \overline{\Sigma}_2\right) }$, $\min_{\overline{\Sigma}_2\times [0,2]}f $ and $\|f\|_{C^{1,1}\left( \overline{\Sigma}_2\times [0,2]\right) }$, such that $B_{20C\epsilon}\subseteq \Sigma_1$ and 
\begin{align*}
\max_{\overline{B}^y_\epsilon} b^*(y)\leq C\int_{B_{2C\epsilon}}b(x)\sigma_k(D^2u(x)) dx+C.
\end{align*}
\end{lemm}

\begin{proof}
By (\ref{g-1}), we have 
\begin{align*}
\sum_i G^{ii}\geq C,
\end{align*}
where $C$ is a positive constant depending only on $n$, $k$, diameter of $\Sigma_2$, $\|u\|_{C^{0,1}\left( \overline{\Sigma}_2\right) }$, $\min_{\overline{\Sigma}_2\times [0,2]}f $ and $\|f\|_{C^{1,1}\left( \overline{\Sigma}_2\times [0,2]\right) }$.

Together with (\ref{g-2}), we have 
\begin{align*}
\sum_i G^{ii}\left( b^*(y)+C|y|^2\right)_{ii}\geq 0.
\end{align*}

By (\ref{g-1}), $G^{ii}$ is uniformly elliptic. By local maximum principle (Theorem 4.8 in \cite{CC}), for any $B_{2\epsilon}^y\subseteq (\Sigma_1)^*$,
we have 
\begin{align*}
\max_{\overline{B}_\epsilon^y}\left( b^*(y)+C|y|^2\right) \leq \frac{C}{|{B}_\epsilon^y|}\int_{B_{2\epsilon}^y}\left( b^*(y)+C|y|^2\right) dy.
\end{align*}

Consequently, 
\begin{align*}
\max_{\overline{B}_\epsilon^y}b^*(y)\leq C\int_{B_{2\epsilon}^y} b^*(y)dy+C,
\end{align*}
where in the previous line, $C$ also depends on $\epsilon$.

By Lemma \ref{Pogolev-1}, $D^2u\geq \frac{I}{C}$ in $\Sigma_1$. Thus $y(x)=Du(x)$ is uniformly monotone, i.e. $|y(x_I)-y(x_{II})|\geq \frac{1}{C}|x_I-x_{II}|$. It follows that $x(B_{2\epsilon}^y)\subseteq B_{2C\epsilon}$. Consequently, 
\begin{align*}
\max_{\overline{B} _\epsilon^y} b^*(y)\leq C \int_{B_{2C\epsilon}} b(x) \det\left( D^2u\right) dx+C\leq C\int_{B_{2C\epsilon}}b\sigma_k dx+C.
\end{align*}

We now show that there exists $\epsilon>0$ such that $B_{2\epsilon}^y \subseteq (\Sigma_1)^*$. In fact, for any $y\in B_{2\epsilon}^y$, we have $|y|=|Du(x)|\leq 2\epsilon$. It follows that $|u(x)-u(0)|\leq C\epsilon$. Consequently, $x\in \Sigma_1$ by choosing $\epsilon$ sufficiently small.

We now check that $B_{20C\epsilon}\subseteq \Sigma_1$. Since $\|u\|_{C^{0,1}} $ is bounded, for any $x\in B_{20C\epsilon}$, $|u(x)-u(0)|\leq C\epsilon$. Consequently, $x\in \Sigma_1$ by choosing $\epsilon$ sufficiently small.

The lemma is now proved.
\end{proof}

\section{Proof of Theorem \ref{theorem-1}}

In this section, we will finish the proof of Theorem \ref{theorem-1}. 

Since $F^{ij}$ is not divergence free, we will choose $H^{ij}=\sigma_kF^{ij}$ to perform the integration by parts as in \cite{Lu23,Lu24}.

Define 
\begin{align}  \label{H1}
H^{ij}=\sigma_kF^{ij}=\sigma_n^{ij}-\frac{\sigma_n\sigma_k^{ij}}{\sigma_k}.
\end{align}

Before we start to prove the main theorem, we first prove an induction lemma. The lemma will help us to simplify the argument in the integration by
parts.

\begin{lemm}
\label{Cauchy} Let $n\geq 2$, $1\leq k<n$. Let $f\in C^{1,1}\left( \overline{\Sigma}_2\times \mathbb{R}\right) $ be a positive function and let $u\in C^4\left( \overline{\Sigma}_2\right) $ be a convex solution of (\ref{main-eqn}). For any $x_0\in \Sigma_1$, suppose that $D^2u$ is diagonalized at $x_0$.  Let $b=\ln\lambda_1+K_0$, where $\lambda_1$ is the largest eigenvalue of $D^2u$ and $-K_0$ is a uniform lower bound of $\ln\lambda_1$. Then at $x_0$, we have 
\begin{align*}
\sum_i |b_i\sigma_l^{ii}|\leq C\epsilon_0\sum_i H^{ii}b_i^2+\frac{C}{\epsilon_0}\sigma_k,\quad \forall 1\leq l\leq k,
\end{align*}
where $\epsilon_0$ is any positive constant, and $C$ is a positive constant depending only on $n$, $k$, diameter of $\Sigma_2$, $\|u\|_{C^{0,1}\left( \overline{\Sigma}_2\right) }$, $\min_{\overline{\Sigma}_2\times [0,2]}f $ and $\|f\|_{C^{1,1}\left( \overline{\Sigma}_2\times [0,2]\right) }$. 
\end{lemm}

\begin{proof}
Without loss of generality, we may assume $\lambda_1\geq \cdots\geq \lambda_n $.

Note that by Lemma \ref{Pogolev-1}, at $x_0$, we have
\begin{align}\label{HHHHH}
\lambda_n\geq C,
\end{align}
where $C$ is a positive constant depending only on $n$, $k$, diameter of $\Sigma_2$, $\|u\|_{C^{0,1}\left( \overline{\Sigma}_2\right) }$, $\min_{\overline{\Sigma}_2\times [0,2]}f $ and $\|f\|_{C^{1,1}\left( \overline{\Sigma}_2\times [0,2]\right) }$. 

By (\ref{H1}) and Lemma \ref{Sigma_k-Lemma-0}, we have 
\begin{align*}  
H^{ii}=\sigma_n^{ii}-\frac{\sigma_n\sigma_k^{ii}}{\sigma_k}=f\cdot\frac{\sigma_k(|i)}{\lambda_i}.
\end{align*}

For $1\leq i\leq k$, we have 
\begin{align*}
H^{ii}=f\cdot\frac{\sigma_k(|i)}{\lambda_i}\geq C\frac{\sigma_{k+1}}{\lambda_i^2}.
\end{align*}

For $i\geq k+1$, we have 
\begin{align*}
H^{ii}=f\cdot\frac{\sigma_{k}(|i)}{\lambda_i}\geq C\frac{\sigma_{k}}{\lambda_i}.
\end{align*}

For $1\leq i\leq l$, by (\ref{HHHHH}) and the fact that $1\leq l\leq k$, we have 
\begin{align*}
|b_i\sigma_l^{ii}|\leq C|b_i|\frac{\sigma_l}{\lambda_i}\leq C\epsilon_0 \frac{\sigma_{k+1}b_i^2}{\lambda_i^2}+\frac{C}{\epsilon_0}\frac{\sigma_l^2}{\sigma_{k+1}}\leq C\epsilon_0 H^{ii}b_i^2+\frac{C}{\epsilon_0} \sigma_l.
\end{align*}

For $l<i\leq k$, by (\ref{HHHHH}) and the fact that $1\leq l\leq k$, we have 
\begin{align*}
|b_i\sigma_l^{ii}|\leq C|b_i|\sigma_{l-1}\leq C\epsilon_0 \frac{\sigma_{k+1}b_i^2}{\lambda_i^2}+\frac{C}{\epsilon_0}\frac{\lambda_i^2\sigma_{l-1}^2}{\sigma_{k+1}}\leq C\epsilon_0 H^{ii}b_i^2+\frac{C}{\epsilon_0}\sigma_l.
\end{align*}

For $i\geq k+1$, by (\ref{HHHHH}) and the fact that $1\leq l\leq k$, we have 
\begin{align*}
|b_i\sigma_l^{ii}|\leq C|b_i|\sigma_{l-1}\leq C\epsilon_0 \frac{\sigma_{k}b_i^2}{\lambda_i}+\frac{C}{\epsilon_0}\frac{\lambda_i\sigma_{l-1}^2}{\sigma_{k}}\leq C\epsilon_0 H^{ii}b_i^2+\frac{C}{\epsilon_0}\sigma_{l-1}.
\end{align*}

Therefore, 
\begin{align*}
\sum_i |b_i\sigma_l^{ii}|\leq C\epsilon_0\sum_i H^{ii}b_i^2+\frac{C}{\epsilon_0}\sigma_{k},\quad \forall 1\leq l\leq k.
\end{align*}

The lemma is now proved.
\end{proof}

Proof of Theorem \ref{theorem-1}.

\begin{proof}
For the sake of simplicity, we will use $C$ to denote a universal constant depending only on $n$, $k$, diameter of $\Sigma_2$, $\|u\|_{C^{0,1}(\overline{\Sigma}_2)}$, $\min_{\overline{\Sigma}_2\times [0,2]}f $ and $\|f\|_{C^{1,1}\left( \overline{\Sigma}_2\times [0,2]\right)}$.

Let $b=\ln \lambda _1+K_0$. Let $\delta >0$ with $B_{10\delta }\subseteq \Sigma _1$. Let $\varphi $ be a cutoff function such that $\varphi =1$ on $B_\delta $ and $\varphi =0$ outside $B_{2\delta }$. By Lemma \ref{Sigma_k-Lemma-0} and the fact that $\sum_{i}(\sigma _{k}^{ij})_{i}=0$, we have 
\begin{align*}
\int_{B_{\delta }}b\sigma _{k}dx\leq & \ \int_{B_{2\delta }}\varphi b\sigma_{k}dx=\frac{1}{k}\sum_{i,j}\int_{B_{2\delta }}\varphi b\sigma_{k}^{ij}u_{ij}dx \\
=& \ -\frac{1}{k}\sum_{i,j}\int_{B_{2\delta }}\left( \varphi _{i}b+\varphi b_{i}\right) \sigma _{k}^{ij}u_{j}dx \\
\leq & \ C\int_{B_{2\delta }}b\sigma _{k-1}dx-\frac{1}{k}\sum_{i,j}\int_{B_{2\delta }}\varphi b_{i}\sigma _{k}^{ij}u_{j}dx.
\end{align*}

By Lemma \ref{Cauchy}, we have 
\begin{equation*}
\left\vert -\frac{1}{k}\sum_{i,j}\int_{B_{2\delta }}\varphi b_{i}\sigma_{k}^{ij}u_{j}dx\right\vert \leq C\sum_{i,j}\int_{B_{2\delta}}H^{ij}b_{i}b_{j}dx+C\int_{B_{2\delta }}\sigma _{k}dx.
\end{equation*}

It follows that 
\begin{equation*}
\int_{B_{\delta }}b\sigma _{k}dx\leq C\int_{B_{\delta }}b\sigma
_{k-1}dx+C\sum_{i,j}\int_{B_{2\delta }}H^{ij}b_{i}b_{j}dx+C\int_{B_{2\delta
}}\sigma _{k}dx.
\end{equation*}

With the help of Lemma \ref{Cauchy}, we can repeat the process and obtain 
\begin{align}\label{p1-1}
\int_{B_{\delta }}b\sigma _{k}dx\leq C\sum_{i,j}\int_{B_{3\delta
}}H^{ij}b_{i}b_{j}dx+C\int_{B_{3\delta }}\sigma _{k}dx.
\end{align}

By (\ref{H1}) and Lemma \ref{Jacobi}, we have 
\begin{align*}
\sum_{i,j} H^{ij}b_{ij}\geq c(n,k)\sum_{i,j} H^{ij}b_ib_j -C\sigma_k,
\end{align*}
in the viscosity sense.

By Theorem 1 in \cite{I}, $b$ also satisfies the above inequality in the
distribution sense. Let $\Phi $ be a cutoff function such that $\Phi =1$ on $%
B_{3\delta }$ and $\Phi =0$ outside $B_{4\delta }$. Then 
\begin{align*}
& \ \sum_{i,j}\int_{B_{3\delta }}H^{ij}b_{i}b_{j}dx\leq\sum_{i,j}\int_{B_{4\delta }}\Phi ^{2}H^{ij}b_{i}b_{j}dx \\
\leq & \ \frac{1}{c(n,k)}\sum_{i,j}\int_{B_{4\delta }}\Phi^{2}H^{ij}b_{ij}dx+C\int_{B_{4\delta }}\sigma _{k}dx \\
\leq & \ -\frac{2}{c(n,k)}\sum_{i,j}\int_{B_{4\delta }}\Phi \Phi_{j}H^{ij}b_{i}dx-C\sum_{i,j}\int_{B_{4\delta }}\Phi^{2}(H^{ij})_{j}b_{i}dx+C\int_{B_{4\delta }}\sigma _{k}dx \\
\leq & \ \frac{1}{2}\sum_{i,j}\int_{B_{4\delta }}\Phi^{2}H^{ij}b_{i}b_{j}dx+C\sum_{i,j}\int_{B_{4\delta }}\Phi _{i}\Phi_{j}H^{ij}dx+C\sum_{i,j}\int_{B_{4\delta }}\Phi ^{2}f_{j}\sigma
_{k}^{ij}b_{i}dx \\
& \ +C\int_{B_{4\delta }}\sigma _{k}dx,
\end{align*}
we have used the fact that $\sum_{j}(H^{ij})_{j}=\sum_{j}\left( (\sigma_{n}^{ij})_{j}-(f\sigma _{k}^{ij})_{j}\right) =-\sum_{j}f_{j}\sigma _{k}^{ij}$ in the last inequality.

Consequently, 
\begin{equation*}
\frac{1}{2}\sum_{i,j}\int_{B_{4\delta }}\Phi ^{2}H^{ij}b_{i}b_{j}dx\leq C\sum_{i,j}\int_{B_{4\delta }}\Phi _{i}\Phi_{j}H^{ij}dx+C\sum_{i,j}\int_{B_{4\delta }}\Phi ^{2}f_{j}\sigma
_{k}^{ij}b_{i}dx+C\int_{B_{4\delta }}\sigma _{k}dx.
\end{equation*}

Together with Lemma \ref{Sigma_k-Lemma-0}, Lemma \ref{Cauchy} and the fact
that $H^{ij}=\sigma _{n}^{ij}-f\sigma _{k}^{ij}$ , we have 
\begin{align*}
& \ \frac{1}{2}\sum_{i,j}\int_{B_{4\delta }}\Phi ^{2}H^{ij}b_{i}b_{j}dx \\
\leq & \ C\sum_{i,j}\int_{B_{4\delta }}\Phi _{i}\Phi _{j}H^{ij}dx+C\epsilon_{0}\sum_{i,j}\int_{B_{4\delta }}\Phi ^{2}H^{ij}b_{i}b_{j}dx+\frac{C}{\epsilon_{0}}\int_{B_{4\delta }}\sigma _{k}dx \\
\leq & \ C\int_{B_{4\delta }}\left( C\sigma _{n-1}-C\sigma _{k-1}\right) dx+\frac{1}{4}\sum_{i,j}\int_{B_{4\delta }}\Phi^{2}H^{ij}b_{i}b_{j}dx+C\int_{B_{4\delta }}\sigma _{k}dx \\
\leq & \ \frac{1}{4}\sum_{i,j}\int_{B_{4\delta }}\Phi^{2}H^{ij}b_ib_jdx+C\int_{B_{4\delta }}\sigma _{n-1}dx,
\end{align*}
we have used the fact that $\lambda _{n}\geq C$ in the last inequality.

Consequently, 
\begin{align*}
\frac{1}{4}\sum_{i,j}\int_{B_{4\delta}} \Phi^2H^{ij}b_ib_jdx\leq C\int_{B_{4\delta}}\sigma_{n-1}dx.
\end{align*}

Plugging into (\ref{p1-1}), we have 
\begin{equation*}
\int_{B_{\delta }}b\sigma _{k}dx\leq C\int_{B_{4\delta }}\sigma _{n-1}dx.
\end{equation*}

Let $\Psi$ be a cutoff function such that $\Psi=1$ on $B_{4\delta}$ and $\Psi=0$ outside $B_{5\delta}$. Then 
\begin{align*}
\int_{B_{4\delta}}\sigma_{n-1}dx\leq \int_{B_{5\delta}}\Psi\sigma_{n-1}dx=-\frac{1}{n-1}\sum_{i,j}\int_{B_{5\delta}}\Psi_i\sigma_{n-1}^{ij}u_jdx\leq C\int_{B_{5\delta}}\sigma_{n-2}dx.
\end{align*}

Thus 
\begin{align*}
\int_{B_\delta}b\sigma_kdx\leq C\int_{B_{5\delta}}\sigma_{n-2}dx.
\end{align*}

Repeating the process, we have 
\begin{align*}
\int_{B_\delta}b\sigma_kdx\leq C\int_{B_{6\delta}}\sigma_1dx\leq C.
\end{align*}

Let $\delta=2C\epsilon$, together with Lemma \ref{mean}, we have 
\begin{align*}
\max_{\overline{B}_\epsilon^y}b^*(y)\leq C.
\end{align*}

It follows that $|D^2u(x(y))|\leq C$ for all $y\in B^y_\epsilon$. Consequently, $D^2w(y)\geq \frac{I}{C}$ for all $y\in B^y_\epsilon$. Thus $x(y)=Dw(y)$ is uniformly monotone, i.e. $|x(y_I)-x(y_{II})|\geq \frac{1}{C}|y_I-y_{II}|$. It follows that $y(B_{\frac{\epsilon}{C}})\subseteq B^y_\epsilon$. Let $\tau=\frac{\epsilon}{C}$, then 
\begin{align*}
\max_{ \overline{B}_\tau} |D^2u|\leq C.
\end{align*}

The theorem is now proved.
\end{proof}

\section{Strict convexity of viscosity solutions}

In this section, we will prove the strict convexity of viscosity solutions by following the method of Urbas \cite{Urbas88}, and Collins and Mooney \cite{Mooney-Collins}.

\subsection{$\alpha>1-\frac{2}{n-k}$ and $p>\frac{( n-1)( n-k) }{2}$}

In this subsection, we will prove strict convexity by following the method of Urbas \cite{Urbas88}.
\begin{lemm}\label{convex}
Let $n\geq 4$, $1\leq k\leq n-3$ and let $\Omega$ be a bounded domain. Let $u$ be viscosity subsolution of 
\begin{equation*}
\frac{\sigma _n}{\sigma _k}\left( D^2u\right) =\mu, \quad in\quad\Omega,
\end{equation*}
where $\mu $ is a positive constant. 

Suppose one of the following holds:
\begin{enumerate}
\item $u\in C^{1,\alpha }\left( \overline{\Omega}\right) $ for some $\alpha >1-\frac{2}{n-k}$;
\item $u\in W^{2,p}( \Omega ) $ for some $p>\frac{( n-1)( n-k) }{2}$.
\end{enumerate}

Then the graph of $u$ is strictly convex. In particular, for any ball $B_R ( x_0) \subseteq \Omega $ and any $0<\rho \leq \frac{R}{4}$, we have
\begin{equation*}
\min_{x\in \partial B_\rho ( x_0) }\left( u( x)-u( x_0) -Du( x_0) \cdot ( x-x_0) \right) \geq C>0,
\end{equation*}
where $C$ is a positive constant depending on $n,k,\alpha ,\mu, R ,\rho ,\| u\|_{C^{1,\alpha }\left( \overline{\Omega}\right) }$ for case $(1)$, and on $n,k,p,\mu ,R,\rho ,\| u\| _{W^{2,p} ( \Omega ) }$ for case $(2)$.
\end{lemm}

Proof of Lemma \ref{convex} for case $(1)$.

\begin{proof}
Without loss of generality, we may assume $x_0=0$. By subtracting a linear function, we may assume $u(0) =0$ and $Du( 0) =0$.  Let $0< \rho \leq \frac{R}{4}$ be fixed. Suppose
\begin{equation*}
\lambda =\min_{\partial B_\rho}u.
\end{equation*}

By rotating the coordinate, we may assume 
\begin{equation*}
u( 0,\cdots ,0,\rho) =\lambda.
\end{equation*}

Let $M$ be a positive constant satisfying
\begin{align*}
\|u\|_{C^{1,\alpha}\left(\overline{\Omega}\right)}\leq M.
\end{align*}

For the sake of simplicity, denote 
\begin{align*}
x^\prime=( x_1,\cdots,x_{n-1}).
\end{align*}

Since $Du\in C^\alpha$, we have
\begin{align*}
|Du(x^\prime,0)|=|Du(x^\prime,0)-Du(0)|\leq M|x^\prime|^\alpha.
\end{align*}

Together with mean value theorem, we have
\begin{align*}
|u(x^\prime,0)|=|u(x^\prime,0)-u(0)|=|Du(\xi,0)||x^\prime|\leq M|x^\prime|^{1+\alpha}.
\end{align*}

It follows that
\begin{align}\label{e1}
u(x^\prime,0)\leq M|x^\prime|^{1+\alpha}.
\end{align}

Since $u(0,\rho)=\min_{\partial B_\rho }u$, we have $D_{x^\prime}u(0,\rho)=0$. Together with the fact that $Du\in C^\alpha$, we have
\begin{align*}
|D_{x^\prime}u(x^\prime,\rho)|=|D_{x^\prime}u(x^\prime,\rho)-D_{x^\prime}u(0,\rho)|\leq M|x^\prime|^\alpha.
\end{align*}

Combining with mean value theorem, we have
\begin{align*}
|u(x^\prime,\rho)-u(0,\rho)|=|D_{x^\prime}u(\xi,\rho)||x^\prime|\leq M|x^\prime|^{1+\alpha}.
\end{align*}

Consequently,
\begin{align*}
|u(x^\prime,\rho)|\leq |u(x^\prime,\rho)-u(0,\rho)|+|u(0,\rho)|\leq M|x^\prime|^{1+\alpha}+\lambda.
\end{align*}

It follows that
\begin{align}\label{e2}
u(x^\prime,\rho)\leq M|x^\prime|^{1+\alpha}+\lambda.
\end{align}

Applying Young's inequality to (\ref{e1}) and (\ref{e2}) with $p=\frac{2}{1+\alpha }$ and $q=\frac{2}{1-\alpha }$, we obtain
\begin{align}\label{e3}
u( x^\prime,0) \leq &\  A|x^\prime|^2+C_0A^{-\gamma },\\ \nonumber
u( x^\prime,\rho ) \leq &\ A|x^\prime|^2+C_0A^{-\gamma}+\lambda,
\end{align}
where $A$ is a positive constant to be determined later, $\gamma =\frac{1+\alpha}{1-\alpha}$ and $C_0$ is a positive constant depending only on $M$ and $\alpha $.

Let
\begin{equation*}
U=\{x\in \Omega:|x^\prime|<\rho,0<x_n<\rho\}\subseteq B_R.
\end{equation*}

Consider the following barrier function
\begin{equation*}
w( x^\prime,x_n) =A|x^\prime|^2+C_0A^{-\gamma }+\frac{\lambda x_n}{\rho }+Bx_n(x_n-\rho),
\end{equation*}
where $0<B\leq 1$ is a constant to be determined later. 

On $\partial U\cap\{x_n=0\}$ and $\partial U\cap \{x_n=\rho\}$, we have $w\geq u$ by (\ref{e3}). 

On $\partial U\cap \{|x^\prime|=\rho\}$, we have
\begin{equation*}
w\geq A\rho ^2+C_0A^{-\gamma }-\frac{B\rho^2}{4}.
\end{equation*}

Choose $A$ sufficiently large such that
\begin{equation*}
A\geq \rho^{-2}\sup_{B_R}u+\frac{1}{4}. 
\end{equation*}

It follows that $w\geq u$ on $\partial U$.

Note that
\begin{align*}
D^2w=diag(2A,\cdots,2A,2B).
\end{align*}

It follows that
\begin{equation*}
\frac{\sigma _n}{\sigma _k}\left( D^2w\right) \leq \frac{2^nA^{n-1}B}{2^kA^k}=2^{n-k}A^{n-k-1}B.
\end{equation*}

Let
\begin{equation*}
B=2^{-(n-k)}A^{-(n-k-1)}\cdot \mu.
\end{equation*}

We may choose $A$ sufficiently large such that $B\leq 1$. Consequently, $\frac{\sigma _n}{\sigma _k}\left( D^2w\right) \leq \mu$. By Lemma \ref{compare 1}, we obtain $w\geq u$ in $U$. In particular,
\begin{equation*}
0\leq u\left( 0,\frac{\rho }{2}\right) \leq  w\left( 0,\frac{\rho }{2}\right) =C_0A^{-\gamma }+\frac{\lambda }{2}-2^{-(n-k+2)}A^{-(n-k-1)}\cdot \mu \rho^2.
\end{equation*}

Since $\alpha >1-\frac{2}{n-k}$, we have $\gamma=\frac{1+\alpha}{1-\alpha} >n-k-1$. Choose $A$ sufficiently large such that
\begin{equation*}
C_0A^{-\gamma }\leq 2^{-(n-k+2)}A^{-(n-k-1)}\cdot \frac{ \mu \rho ^2}{2}.  
\end{equation*}

It follows that
\begin{equation*}
0\leq \frac{\lambda }{2}-2^{-(n-k+2)}A^{-(n-k-1)}\cdot \frac{ \mu \rho ^2}{2}. 
\end{equation*}

Consequently,
\begin{align*}
\lambda\geq 2^{-(n-k+2)}A^{-(n-k-1)}\cdot \mu \rho^{2}.
\end{align*}

The lemma for case $(1)$ is now proved.
\end{proof}

Proof of Lemma \ref{convex} for case $(2)$.
\begin{proof}
Without loss of generality, we may assume $x_0=0$. By subtracting a linear function, we may assume $u(0)=0$ and $Du( 0) =0$.  Let $0< \rho \leq \frac{R}{4}$ be fixed. Suppose
\begin{equation*}
\lambda =\min_{\partial B_\rho }u.
\end{equation*}

By rotating the coordinate, we may assume 
\begin{equation*}
u( 0,\cdots,0,\rho) =\lambda.
\end{equation*}

Let $K$ be a positive constant satisfying
\begin{equation*}
\| u\|_{W^{2,p}( \Omega) }\leq K.
\end{equation*}

For the sake of simplicity, denote 
\begin{align*}
x^\prime=( x_1,\cdots,x_{n-1}).
\end{align*}

Since $u\in W^{2,p}$, by Fubini theorem, we have
\begin{equation*}
\sum_{|\alpha|\leq 2}\int_{0}^{\frac{\rho }{4}}\int_{B_{2\rho }^{n-1}}| D_{x^\prime}^\alpha u\left( x^\prime,x_n\right)| ^p dx^\prime dx_n\leq K^p,
\end{equation*}
where $B_{2\rho }^{n-1}$ is the ball with radius $2\rho$ in $\mathbb{R}^{n-1}$.

By mean value theorem for integrals, there exists $0< \rho_1<\frac{\rho}{4}$ such that the function $v_1$ defined by
\begin{align*}
v_1(x^\prime) =u(x^\prime,\rho _1)
\end{align*}
satisfies
\begin{equation*}
\| v_{1}\|_{W^{2,p}\left( B_{2\rho}^{n-1}\right) }^p\leq \frac{4K^p}{\rho}.
\end{equation*}

Similarly, there exists $\frac{3\rho}{4}<\rho_2<\rho$ such that the function $v_2$ defined by
\begin{align*}
v_2(x^\prime) =u(x^\prime,\rho _2)
\end{align*} 
satisfies
\begin{equation*}
\| v_2\|_{W^{2,p}\left( B_{2\rho}^{n-1}\right) }^p\leq \frac{4K^p}{\rho}.
\end{equation*}

By Sobolev embedding theorem, we have 
\begin{equation} \label{w1}
\|v_i\| _{C^{1,\alpha }\left( \overline{B}_\rho^{n-1}\right) }\leq M,\quad i=1,2,
\end{equation}
where $M$ is a positive constant depending only on $n,p,\rho,K,R$, and $\alpha=1-\frac{n-1}{p}>1-\frac{2}{n-k}$.

Since $u$ is convex, $u(0) =0$ and $Du( 0) =0$, it follows that
\begin{equation}\label{w2}
u(0,\rho_1)\leq \lambda,\quad u(0,\rho_2)\leq \lambda.
\end{equation}

By convexity of $u$, we have
\begin{align}\label{w3}
\sup_{B_{\frac{R}{2}}}|Du|\leq \frac{8}{R}\sup_{B_{\frac{3R}{4}}}u\leq C_1,
\end{align}
where $C_1$ is a positive constant depending only on $n,p,K,R$.

By (\ref{w1}) and mean value theorem, we have
\begin{align*}
u(x^\prime,\rho_1)-u(0,\rho_1)=v_1(x^\prime)-v_1(0)=Dv_1(\xi)\cdot x^\prime\leq  Dv_1(0)\cdot x^\prime+M|x^\prime|^{1+\alpha}.
\end{align*}

Together with (\ref{w2}) and (\ref{w3}), we have
\begin{align}\label{w4}
u(x^\prime,\rho _1) \leq  \lambda+\sum_{i=1}^{n-1}a_ix_i+M| x^\prime|^{1+\alpha},
\end{align}
where $Dv_1(0)=a=(a_1,\cdots,a_{n-1})$ satisfies $|a|\leq C_1$.

Similarly,
\begin{align}\label{w5}
u(x^\prime,\rho _2) \leq  \lambda+\sum_{i=1}^{n-1}b_ix_i+M| x^\prime|^{1+\alpha},
\end{align}
where $Dv_2(0)=b=(b_1,\cdots,b_{n-1})$ satisfies $|b|\leq C_1$.

Applying Young's inequality to (\ref{w4}) and (\ref{w5}), we obtain
\begin{align}\label{a1}
u(x^\prime,\rho _1) \leq &\ \lambda +\sum_{i=1}^{n-1}a_ix_i+A|x^\prime|^{2}+C_0A^{-\gamma },\\ \nonumber
u( x^\prime,\rho _2)\leq &\ \lambda +\sum_{i=1}^{n-1}b_ix_i+A|x^\prime|^{2}+C_0A^{-\gamma },
\end{align}
where $A$ is a positive constant to be determined later, $\gamma =\frac{1+\alpha}{1-\alpha}$ and $C_0$ is a positive constant depending on $M$ and $\alpha$.

Let
\begin{equation*}
U=\{x\in \Omega:|x^\prime|<\rho,\rho_1<x_n<\rho_2\}\subseteq B_R.
\end{equation*}

Consider the following barrier function
\begin{align*}
w(x^\prime,x_n) =&\  A| x^\prime|^2+C_0A^{-\gamma}+B( x_n-\rho _1) ( x_n-\rho _2) \\
&\ +\left( \lambda +\sum_{i=1}^{n-1}( b_i-a_i) x_i\right) \frac{x_n-\rho _1}{\rho _2-\rho _1}+\lambda+\sum_{i=1}^{n-1}a_ix_i
\end{align*}
where $0<B\leq 1$ is a constant to be determined later. 

On $\partial U\cap\{x_n=\rho_1\}$ and $\partial U\cap \{x_n=\rho_2\}$, we have $w\geq u$ by (\ref{a1}). 

On $\partial U\cap \{|x^\prime|=\rho\}$, we have
\begin{equation*}
w\geq A\rho ^2+C_0A^{-\gamma }-\frac{B}{4}( \rho _2-\rho_1) ^2-3C_1\rho.
\end{equation*}

Choose $A$ sufficiently large such that
\begin{equation*}
A\geq \rho ^{-2}\sup_{B_R}u+\frac{1}{4}+\frac{3C_1}{\rho }.
\end{equation*}

It follows that $w\geq u$ on $\partial U$.

Note that 
\begin{equation*}
D^2w=diag\left( 2A,\cdots, 2A,2B\right) +(a_{ij}),
\end{equation*}
where 
\begin{align*}
a_{in}=a_{ni}=\frac{b_i-a_i}{\rho _2-\rho _1},\quad  1\leq i\leq n-1,
\end{align*}
and all other $a_{ij}$ are zero.

By direct computation, we have
\begin{equation*}
\sigma _k\left( D^2w\right) \geq 2^kA^k-C(n,k)A^{k-2}\max_{1\leq i\leq n-1}a_{in}^2.
\end{equation*}

Note that
\begin{equation*}
a_{in}^{2}\leq \frac{16C_1^2}{\rho ^2}.
\end{equation*}

Choose $A$ sufficiently large such that
\begin{equation*}
2^{k-1}A^k>C(n,k)A^{k-2}\cdot \frac{16C_1^2}{\rho ^2}>0.
\end{equation*}

It follows that 
\begin{equation*}
\frac{\sigma _{n}}{\sigma _{k}}\left( D^{2}w\right) \leq \frac{2^nA^{n-1}B}{2^{k-1}A^{k}}< 2^{n-k+1}A ^{n-k-1}B.
\end{equation*}

Let
\begin{equation*}
B=2^{-(n-k+1)}A^{-(n-k-1)}\cdot \mu.
\end{equation*}

We may choose $A$ sufficiently large such that $B\leq 1$. Consequently, $\frac{\sigma _n}{\sigma _k}\left( D^2w\right) \leq \mu $. By Lemma \ref{compare 1}, we obtain $w\geq u$ in $U$. In particular,
\begin{align*}
0\leq u\left( 0,\frac{\rho }{2}\right)\leq  w\left( 0,\frac{\rho }{2}\right) =&\ C_0A^{-\gamma }+2^{-(n-k+1)}A^{-(n-k-1)}\cdot \mu\left( \frac{\rho}{2}-\rho _1\right) \left( \frac{\rho}{2}-\rho_2\right) \\
&\ +\lambda \left( 1+\frac{\frac{\rho}{2}-\rho _1}{\rho _2-\rho _1}\right).
\end{align*}

Since $\alpha >1-\frac{2}{n-k}$, we have $\gamma >n-k-1$. We may choose $A$ sufficiently large such that
\begin{equation*}
C_0A^{-\gamma }<2^{-(n-k+1)}A^{-(n-k-1)}\cdot \frac{\mu\rho ^2}{32}. 
\end{equation*}

Together with the fact that $0<\rho_1<\frac{\rho}{4}$ and $\frac{3\rho}{4}<\rho_2<\rho$, we have
\begin{equation*}
0 \leq -2^{-(n-k+1)}A^{-(n-k-1)}\cdot\frac{\mu \rho^2}{32}+2\lambda .
\end{equation*}

It follows that
\begin{equation*}
\lambda \geq 2^{-(n-k+1)}A^{-(n-k-1)}\cdot\frac{\mu\rho^2 }{64}.
\end{equation*}

The lemma for case $(2)$ is now proved.
\end{proof}

\subsection{$p=\frac{( n-1) (n-k)}{2}$}

In this subsection, we will prove strict convexity by following the method of Collins and Mooney \cite{Mooney-Collins}.
\begin{lemm}
\label{RealMain} 
Let $n\geq 4$, $1\leq k\leq n-3$. Let $u$ be a viscosity subsolution of
\begin{align*}
\frac{\sigma _{n}}{\sigma _{k}}\left( D^{2}u\right) = 1,\quad in \quad B_1.
\end{align*}

Assume $ u\in W^{2,p}\left( B_{1}\right)$ for some $p\geq \frac{(n-1) (n-k)}{2}$. 

Then the graph of $u$ is strictly convex. In particular, there exists $h_0>0$ depending only on $n$, $k$, $p$ and $\|u\|_{W^{2,p}(B_1)}$ such that
\begin{equation*}
\left\{ x\in B_1| u ( x) <u( x_0) +q\cdot (x-x_0) +h_0 \right\} \Subset  B_1
\end{equation*}
for all $x_0\in B_{\frac{1}{2}}$ and some subgradient $q$ at $x_0$.
\end{lemm}

To prove the above result, we will need a few technical lemmas first. 
\begin{lemm}
\label{RadiusEstimate}
Let $n\geq 3$, $1\leq k<n$ and let $\Omega$ be a bounded domain. Let $u$ be a viscosity subsolution of
\begin{align*}
\frac{\sigma _{n}}{\sigma _{k}}\left( D^{2}u\right)=1,\quad in \quad\Omega.
\end{align*}

Assume further that $u\geq 0$ in $\Omega$ and $u(0)=0$.

Denote the section $\Sigma _h=\left\{ x\in \Omega |u( x)<h\right\} $. Let $E$ be the ellipsoid of maximal volume contained within $\Sigma _{h}$. Let $a_1\geq\cdots\geq a_n$ be the length of principal semi-axes of $E$.

Then 
\begin{equation*}
(2h)^{\frac{n-k}{2}}\geq  a_{1}\cdots a_{n-k} .
\end{equation*}
\end{lemm}

\begin{proof}
Replace $u$ with $u-h$, we have 
\begin{equation*}
u|_{\partial \Sigma _h}\leq 0,\quad \left\vert \min_{\Sigma_h}u\right\vert =h.
\end{equation*}

By rotation, we may assume 
\begin{equation}
E=\left\{ x\in 
\mathbb{R}
^{n}\left\vert \frac{( x_{1}-r_{1}) ^{2}}{a_{1}^{2}}+\cdots+\frac{( x_{n}-r_{n}) ^{2}%
}{a_{n}^{2}}\leq 1\right. \right\} ,  \label{JE}
\end{equation}
where $a_{1}\geq a_{2}\geq ...\geq a_{n}$. Let $A=diag(a_{1}^{2},\cdots ,a_{n}^{2})$. 

Define
\begin{equation*}
w=\left( \sigma _{n-k}( A) \right) ^{\frac{1}{n-k}}\cdot \frac{1}{2}\cdot \left( \frac{( x_{1}-r_1) ^{2}}{a_{1}^{2}}+\cdots+\frac{( x_{n}-r_{n}) ^{2}}{a_{n}^{2}}-1\right).
\end{equation*}

It is easy to check that $\frac{\sigma _{n}}{\sigma _{k}}\left( D^{2}w\right) =1$ and $w\geq u$ on $\partial E$. 

By Lemma \ref{compare 1}, $w\geq u$ in $E$. Consequently,
\begin{equation*}
h\geq |u( r_{1},r_{2},...,r_{n}) |\geq  \frac{1}{2}\left(\sigma _{n-k}\left( A\right) \right) ^{\frac{1}{n-k}}\geq \frac{1%
}{2}\left( a_{1}\cdots a_{n-k}\right) ^{\frac{2}{n-k}}.
\end{equation*}

The lemma is now proved.
\end{proof}

Denote $x=(x^\prime,x_n)\in \mathbb{R}^{n}$, where $x^\prime=(x_1,\cdots,x_{n-1})$. 

\begin{lemm}
\label{SingularityGrowth} 
Let $n\geq 3$, $1\leq k<n-1$. Let $u$ be a viscosity subsolution of
\begin{align*}
\frac{\sigma _{n}}{\sigma _{k}}\left(D^{2}u\right)= 1,\quad in\quad \{|x^\prime|<1\}\cap \{|x_n|<1\}.
\end{align*}

Assume further that $u\geq 0$ in $\{|x^\prime|<1\}\cap \{|x_n|<1\}$, and $u=0$ on $\{x^\prime=0\}$.

Then for all $0<r<1$ we have 
\begin{equation*}
\inf_{|x_n|<1}\sup_{|x^\prime|=r}u(x^\prime,x_n)\geq c(n,k)r^{2-\frac{2}{n-k}},
\end{equation*}
where $c(n,k)$ is a positive constant depending only on $n$ and $k$.
\end{lemm}

\begin{proof}
Assume this is not true, then there exists $0<r_0<1$ and $|\tilde{x}_n|<1$ such that
\begin{align*}
\sup_{|x^\prime|=r_0}u(x^\prime,\tilde{x}_n)<c(n,k)r_0^{2-\frac{2}{n-k}}.
\end{align*}

Denote $c=c(n,k)$ and define $h=cr_{0}^{2-\frac{2}{n-k}}$. Then we have 
\begin{equation*}
G=\{x_n=\tilde{x}_n\}\cap \left\{ |x^\prime|<(c^{-1}h)^{\frac{1}{2}\frac{n-k}{n-k-1}%
}\right\} \subset \Sigma _{h},
\end{equation*}
where $\Sigma _h=\left\{  u<h\right\} $.

Since $\Sigma _{h}$ is convex, it contains the convex hull of the set $G$ and $\{x^\prime=0\}\cap \{|x_n|<1\}$. We conclude that $\Sigma _{h}$ contains the
following ellipsoid $F$.%
\begin{equation*}
F=\left\{ x\in 
\mathbb{R}
^{n}\left\vert \frac{x_{1}^{2}}{b_{1}^{2}}+\cdots+\frac{x_{n-1}^{2}}{b_{n-1}^{2}}+\frac{\left( x_{n}-\tilde{r}_n\right)
^{2}}{b_{n}^{2}}\leq 1\right. \right\} ,
\end{equation*}%
where $b_n=\frac{1}{4}\left( | \tilde{x}_n| +1\right) $, $b_i=\frac{1}{2n}(c^{-1}h)^{\frac{1}{2}\frac{n-k}{n-k-1}}$ for $1\leq i\leq n-1
$, and $\tilde{r}_n$ is the midpoint between $\tilde{x}_n$ and the further endpoint of $|x_n|=1$. 

Let $E$ be the ellipsoid of maximal volume contained within $\Sigma _{h}$. We may translate $\Sigma _{h}$ such that $0$ is the center of
mass. Then by John's lemma (see Lemma 2.1 in \cite{Mooney2015}), we have
\begin{equation*}
E\subset \Sigma _{h}\subset C( n) E.
\end{equation*}

Let $a_1\geq\cdots\geq a_n$ be the length of principal semi-axes of $E$. Since $F\subset \Sigma _{h}\subset C( n) E$, we
have
\begin{equation*}
a_{1}\cdots a_{n-k}\geq \tilde{C}( n,k) b_{1}\cdots b_{n-k-1}\cdot b_n.
\end{equation*}

Together with lemma \ref{RadiusEstimate}, we have
\begin{equation*}
(2h)^{\frac{n-k}{2}}\geq \tilde{C}(n,k) b_{1}\cdots b_{n-k-1}\geq \hat{C}(n,k)(c^{-1}h)^{\frac{n-k}{2}},
\end{equation*}%
which is a contradiction by choosing $c=c(n,k)$ sufficiently small.
\end{proof}

\begin{proof}[\textbf{Proof of Lemma \ref{RealMain}:}]
Assume that $u$ is not strictly convex. In other words, $u$ agrees with a
tangent plane on a set of dimension at least $1$. After subtracting the
tangent plane, translating and rescaling we may assume that $u\geq 0$ in $%
\{|x^\prime|<1\}\cap \{|x_n|<1\}$, and that $u=0$ on $\{x^\prime=0\}$. By Lemma \ref{SingularityGrowth}, we also have that 
\begin{equation*}
\inf_{|x_n|<1}\sup_{|x^\prime|=r}u(x^\prime, x_n)\geq c(n,k)r^{2-\frac{2}{n-k}}.
\end{equation*}

Apply Lemma 2.2 in \cite{Mooney-Collins} on the slices $\{x_n=const\}$ (taking $q=2-\frac{2}{n-k}$ and replacing $n$ by $n-1$) and integrate in $x_n$ to
conclude that 
\begin{equation*}
\int_{\{|x_n|<1\}\cap \{r<|x^\prime|<1\}}(\Delta u)^{\frac{( n-1) (n-k)}{2}}dx\geq c(n,k)|\log r|.
\end{equation*}
Since $u\in W^{2,p}$ for $p\geq \frac{( n-1) (n-k)}{2}$. Taking $r\rightarrow 0$ completes the proof of $u$ being strictly convex. The
existence of $h_{0}$ follows directly from Lemma 5.2 in \cite{Mooney-Collins}.
\end{proof}

\section{Proof of Theorem \ref{theorem-v}}

In this section, we will prove Theorem \ref{theorem-v}.

\begin{proof}
Let $x_0\in \Omega$, suppose $B_{8r}( x_0) \subseteq \Omega $. Without loss of generality, we may assume $x_0=0$. By subtracting a linear function, we may assume $u(0)=0$ and $Du( 0) =0$. 

By Sobolev embedding theorem, $u\in C^{1,\gamma }\left( \overline{B}_{6r}\right) $ for some $\gamma$ for case (2). Together with the fact that $u\in C^{1,\alpha}\left(\overline{\Omega}\right)$ in case (1), we may assume $u\in C^{1,\gamma}\left( \overline{B}_{6r}\right) $ in both cases.

Let $\tilde{\varphi}_m $ be a sequence of smooth functions converging to $u$ in uniformly in $\overline{B}_{4r}$. Without loss of generality, we may assume
\begin{align*}
|\tilde{\varphi}_m-u|\leq \frac{1}{2m},\quad in\quad \overline{B}_{4r}.
\end{align*}

Let 
\begin{align*}
\varphi_m=\tilde{\varphi}_m+\frac{1}{2m}.
\end{align*}

It follows that $\varphi_m$ is a sequence of smooth functions converging to $u$ in uniformly in $\overline{B}_{4r}$ and satisfies
\begin{align}\label{weak-1}
\varphi_m\geq u\geq \varphi_m-\frac{1}{m},\quad in\quad \overline{B}_{4r}.
\end{align}

Similarly, let $\title{f}_m$ be a sequence of positive smooth functions converge to $f$ in uniformly in $\overline{B}_{4r}$.  Without loss of generality, we may assume
\begin{align}\label{weak-2}
|\tilde{f}_m-f|\leq \frac{1}{2m},\quad \|\tilde{f}_m\|_{C^{1,1}\left(\overline{B}_{4r}\right)}\leq 2\|f\|_{C^{1,1}\left(\overline{B}_{4r}\right) },\quad in\quad \overline{B}_{4r}.
\end{align}

Let
\begin{align*}
f_m=\tilde{f}_m-\frac{1}{2m}.
\end{align*}

It follows that $f_m$ is a sequence of smooth functions converging to $f$ in uniformly in $\overline{B}_{4r}$ and satisfies
\begin{align}\label{weak-3}
f_m+\frac{1}{m}\geq f\geq f_m,\quad in\quad \overline{B}_{4r}.
\end{align}

When $m$ is sufficiently large, $f_m>0$. Then by \cite{Trudinger}, there exists a $C^{\infty }\left( \overline{B}_{2r} \right) $ convex solution $u_m$ to the following Dirichlet problem
\begin{align*}
\begin{cases}
\frac{\sigma _n}{\sigma _k}\left( D^2u_m\right)=f_m,\quad &\ in\quad B_{2r},\\
u_m = \varphi _m, &\ on\quad \partial B_{2r}.
\end{cases}
\end{align*}

By (\ref{weak-3}), $u_m$ is a supersolution of (\ref{viscosity}). By (\ref{weak-1}) and Lemma \ref{compare 1},
\begin{align*}
u_m\geq u,\quad in \quad B_{2r}.
\end{align*}

Consider
\begin{equation*}
v_m=u_m+\frac{M}{2m}\left(|x|^{2}-4r^{2}\right) -\frac{1}{m},
\end{equation*}
where $M$ is a positive constant to be determined later.  

Let $F=\frac{\sigma _{n}}{\sigma _{k}}$, $\tilde{F}=F^{\frac{1}{n-k}}$,  $%
A=D^{2}u_{m}$ and $B=\frac{M}{m}I$. Since $\tilde{F}$ is concave, we have $2%
\tilde{F}\left( \frac{A+B}{2}\right) \geq \tilde{F}\left( A\right) +\tilde{F}%
\left( B\right) $. Together with the fact that $\tilde{F}$ is homogeneous of degree one, we have
\begin{equation*}
\tilde{F}\left( D^{2}v_{m}\right) \geq \ \tilde{F}\left( D^{2}u_{m}\right) +%
\frac{M}{m}\tilde{F}(I)\geq f_{m}^{\frac{1}{n-k}}+\frac{M}{m}\tilde{F}(I).
\end{equation*}

We may choose $M$ sufficiently large (depending on $n,k$ and $\min_{%
\overline{\Omega}}f$), such that 
\begin{align*}
F(D^2v_m)\geq f_m+\frac{1}{m}\geq f.
\end{align*}

Note that $v_m=\varphi _m-\frac{1}{m}\leq u$ on $\partial B_{2r}$. By Lemma %
\ref{compare 1}, we conclude that 
\begin{align*}
u\geq v_m,\quad in\quad B_{2r}.
\end{align*}

Therefore, 
\begin{equation*}
u_m\geq u\geq u_m+\frac{M}{2m}\left( |x|^2-4r^2\right) -\frac{1}{m},\quad in
\quad B_{2r}.
\end{equation*}

It follows that $u_{m}$ converges to $u$ uniformly in $\overline{B}_{2r}$.

In the following proof, $C$, $\epsilon $ and $\tau $ are positive constants
depending on some of the quantities $n$, $\alpha ,$ $r,$ $k,$ $p,$ $\Omega ,$
$\min_{\overline{\Omega }}f,$ $\left\vert \left\vert f\right\vert
\right\vert _{C^{1,1}\left( \overline{\Omega }\right)}$, $\left\vert
\left\vert u\right\vert \right\vert _{C^{1,\alpha }\left( \overline{\Omega }%
\right) }$ and $\left\vert \left\vert u\right\vert \right\vert
_{W^{2,p}\left( \Omega \right) }$.

Since $u_{m}$ is convex, $u(0)=0$, $Du(0)=0$, and $u\in C^{1,\gamma }$, for
all $m$ sufficiently large, we have 
\begin{equation}
\max_{\overline{B}_{r}}|Du_{m}|\leq \frac{4}{r}\max_{\overline{B}%
_{2r}}|u_{m}|\leq \frac{8}{r}\max_{\overline{B}_{2r}}|u|\leq Cr^{\gamma }.
\label{weak-4}
\end{equation}

Note that for any $r>0$, there exists $N\left( r\right) >0$ such that (\ref%
{weak-4}) holds when $m>N\left( r\right) $. It follows that 
\begin{equation*}
\lim_{m\rightarrow \infty }Du_{m}(0)=0.
\end{equation*}

Define 
\begin{align*}
\overline{u}_m(x)=u_m(x)-u_m(0) -Du_m(0) \cdot x.
\end{align*}

It follows that $\overline{u}_m$ converges to $u$ uniformly in $\overline{B}
_{2r}$ and satisfies 
\begin{align*}
\overline{u}_m(0)=0,\quad D\overline{u}_m(0)=0.
\end{align*}

By Lemma \ref{convex} or Lemma \ref{RealMain}, we have 
\begin{align*}
\min_{x\in \partial B_r}u( x) \geq 2\epsilon>0.
\end{align*}

Denote 
\begin{equation*}
\Sigma^m _\epsilon=\{x\in \mathbb{R}^{n}|\overline{u}_m(x)<\epsilon \}
\end{equation*}

Since $\overline{u}_{m}$ converges to $u$, it follows that for all $m$
sufficiently large, 
\begin{equation}
\Sigma _{\epsilon }^{m}\subseteq B_{r}.  \label{weak-5}
\end{equation}

By (\ref{weak-2}), (\ref{weak-4}) and (\ref{weak-5}), diameter of $\Sigma^m
_\epsilon$, $\max_{\overline{\Sigma}^m _\epsilon}| D\overline{u}_m|$ and $\|
f_m\| _{C^{1,1}\left(\overline{\Sigma}^m _\epsilon \right) }$ are uniformly
bounded. By Theorem \ref{theorem-1}, there exists $\overline{B}_{2\tau
}\subseteq \Sigma_\epsilon^m\subseteq B_r$ such that 
\begin{equation*}
\max_{ \overline{B}_{2\tau} }|D^2\overline{u}_m|\leq C.
\end{equation*}

By Evans-Krylov theorem and Schauder theorem, 
\begin{equation*}
\| \overline{u}_m\| _{C^{3,\beta }\left( \overline{B}_{\tau }\right) }\leq C,
\end{equation*}
for any $0<\beta <1$.

Since $\overline{u}_{m}$ converges to $u$ uniformly in $\overline{B}_{2r}$,
it follows that 
\begin{equation*}
\Vert u\Vert _{C^{3,\beta }\left( \overline{B}_{\tau }\right) }\leq C,
\end{equation*}%
where $C$ also depends on $\beta $ in the previous line.

Theorem \ref{theorem-v} is now proved by a compactness argument.
\end{proof}

\noindent

{\it Acknowledgement}: The first author would like to thank Professor Pengfei Guan for enlightening conversations over concavity of Hessian quotient operator and sharing Lemma \ref{Concavity lemma}. Both authors would like to thank Connor Mooney for suggesting the method in \cite{Mooney-Collins}, which leads to the borderline case in Theorem \ref{theorem-v} (2). This was not included in the original version.

\end{document}